\documentclass[11pt,paper=a4,bibliography=totoc]{scrartcl}

\PassOptionsToPackage{shortlabels}{enumitem}

\usepackage[mathfont=stix2,modernmargins]{mystyle-koma-enhanced}
\usepackage{comment}
\usepackage{soul}

\usetikzlibrary{backgrounds,patterns,positioning,shapes.geometric}

\newcommand{\N}{\mathbb{N}}

\declaretheorem[style=mystyleplain,name=Theorem A,numbered=no]{theoremA}
\declaretheorem[style=mystyleplain,name=Theorem B,numbered=no]{theoremB}
\newenvironment{prop}{\begin{proposition}}{\end{proposition}}

\makeatletter
\let\orgdescriptionlabel\descriptionlabel
\renewcommand*{\descriptionlabel}[1]{%
  \let\orglabel\label
  \let\label\@gobble
  \phantomsection
  \edef\@currentlabel{#1}%
  \let\label\orglabel
  \orgdescriptionlabel{#1}%
}

\renewcommand{\P}{\mathbb{P}}

\newcommand{\1}{\mathbf{1}}

\newcommand{\E}{\mathbb{E}}

\title{Loop vs.\ Bernoulli percolation on trees:\\ strict inequality of critical values}

\author{
Andreas Klippel 
\thanks{Fachbereich Mathematik, Technische Universit\"at Darmstadt, Schlossgartenstrasse 7, 64289 Darmstadt, Germany}\\ andreas.klippel@tu-darmstadt.de
\and
Benjamin Lees \orcidlink{0000-0003-2657-5358}
\thanks{School of Mathematics, University of Leeds, Leeds, LS2 9JT, UK}\\b.t.lees@leeds.ac.uk \\
\and
Christian M\"{o}nch \orcidlink{0000-0002-6531-6482}\thanks{Independent Researcher, 64289 Darmstadt, Germany} \\ cmoench25@gmail.com
}

\date{\today}

\begin{document}
\maketitle

\begin{abstract}
We study loop ensembles on locally finite rooted trees. The loops are induced by a Poisson process of links on the edges; the transposition-only case is the random interchange process, and the same framework covers loop representations of quantum spin systems. The set of edges carrying at least one untyped link forms i.i.d.\ Bernoulli bond percolation with retention probability \(1-e^{-\beta}\), so an infinite loop can occur only if the corresponding link cluster is infinite. Our main results for Galton--Watson trees show both phenomena: if the offspring distribution has finite mean \(m\in(1,\infty)\), then, conditioned on survival, the loop threshold is strictly larger than the link threshold; in contrast, in the random interchange case a heavy-tail
condition on the offspring distribution forces the loop and link thresholds,
averaged over both the tree and the link configuration, to coincide at zero. The
separation theorem follows from a stronger deterministic criterion based on
local loop configurations that cut descendant subtrees out of link clusters.
\end{abstract}

\section{Introduction}
Recent years have seen increasing interest in the cycle and permutation structure of the random stirring model that was introduced by Harris \cite{harris_markov_lattices_1972} in the early 1970s. Under the name \emph{random interchange process}, this model has attracted attention for its connection to the quantum Heisenberg model. It is also valued as a beautiful probabilistic method for studying permutations with geometric constraints.

On a graph \( G = (V, E) \), the random interchange process is defined as follows: Each vertex of the graph initially hosts a labelled particle, and each edge of the graph is equipped with an independent Poisson process of rate $\beta>0$ on $[0,1)$. Whenever the Poisson clock associated with an edge rings, the two particles at the endpoints of that edge swap positions. This process induces a random permutation \( \pi_\beta: V \to V \) for any \( \beta > 0 \).

A long-standing conjecture attributed to T\'oth, as discussed in \cite{elboim2024infinitecyclesinterchangeprocess}, is that for \(\mathbb{Z}^d\) with \(d \ge 3\) and sufficiently large \(\beta\), the permutation \(\pi_\beta\) contains infinite cycles almost surely. A major advance on this conjecture is the recent work of Elboim and Sly, who prove that infinite cycles exist almost surely for $d\geq5$.

Observe that in order to obtain an infinite cycle, it is necessary to have at least one Poisson point on each edge used by the cycle. This observation gives a necessary condition for the existence of infinite loops: there must be an infinite cluster in the Bernoulli edge percolation with retention probability \(1-e^{-\beta}\) induced by the Poisson process. The main question we investigate is how far this necessary condition is from being sufficient on trees.

A remarkable result by Schramm \cite{schramm_compositions_2005} shows that on the complete graph \(K_n\), where edge rates are \(\beta/n\), the existence of macroscopic cycles and the emergence of a giant percolation component are asymptotically equivalent as \(n \to \infty\), implying that their critical parameters coincide. Moreover, it is shown that the normalised ordered sizes of the loops in the giant component exhibit a Poisson-Dirichlet \(\operatorname{PD}(1)\) structure.

In contrast to Schramm's result, M\"uhlbacher \cite{muehlbacher_critical_2021} showed that on infinite graphs of bounded degree, the two critical parameters differ. An explicit lower bound on the size of this difference was recently provided by Betz et al.\ \cite{betz2024looppercolationversuslink}.

Thus, whether percolation is also a sufficient condition depends on the underlying graph. On Galton--Watson trees we prove two complementary results:

\noindent \emph{For supercritical Galton--Watson trees with finite mean offspring, the critical values of the random loop model and of induced link percolation differ.}

\noindent \emph{For a class of sufficiently heavy-tailed offspring laws in the random interchange case, the loop and link thresholds coincide at zero after averaging over both the tree and the link configuration.}

These are Theorems~A and~B below. The proof of Theorem~A uses a stronger deterministic tree criterion, but the criterion is best stated after the technical notation has been introduced. Heuristically, certain local loop configurations cut off descendant subtrees from the link cluster. If this cutting mechanism occurs often enough along the relevant infinite directions of the tree, then the resulting link cluster becomes subcritical. For finite-mean Galton--Watson trees, a large-deviation estimate supplies the required frequency.

We work with the random loop model, which extends the random interchange process. The random loop model was introduced as a representation of certain quantum spin systems.  The random stirring model gives a representation of the ferromagnetic Heisenberg model, where points of the Poisson process are referred to as \emph{crosses}. This representation was introduced by Powers \cite{powers_heisenberg} and then used by T\'oth \cite{toth_pressure} to give a lower bound on the pressure. There is also a representation of the anti-ferromagnetic Heisenberg model proposed by Aizenman and Nachtergaele \cite{aizenman_nachtergaele}.
This representation also involves independent Poisson point processes of intensity \(\beta\) assigned to the edges of \(G\). However, unlike the random stirring model, the points of these Poisson processes, referred to as \emph{bars}, do not correspond to transpositions, and the resulting random structure is richer than a permutation.

Ueltschi \cite{ueltschi_random_loops} observed that these representations can be unified to describe XXZ-models. In this setting, the intensity of \emph{crosses} is given by \(u\beta\), while the intensity of \emph{bars} is \((1-u)\beta\) for some \(u \in [0,1]\). We collectively refer to crosses and bars as \emph{links}. By following these links along \(G \times [0,1)_{\text{per}}\), we uncover a set of \emph{loops} \(\mathcal{L}\); see Figure \ref{fig: loop example} for an illustration.

The connection to spin-\(S\) quantum systems requires an exponential tilt of the loop distribution by \(\theta^{\sharp\mathcal{L}}\) with \(\theta = 2S + 1\). However, the model remains of interest even for the case $\theta=1$, which is the case treated here.

The central question in the study of the random loop model remains whether a phase transition occurs. This problem is particularly challenging to analyse due to the inherent dependencies in the model. On the lattice $\mathbb{Z}^d$, Ueltschi \cite{ueltschi_random_loops} has shown via reflection positivity that for integer $\theta \geq 2$ and for $u \in \bigl[0,\tfrac{1}{2}\bigr]$ a phase transition indeed occurs. These parameter ranges have been improved recently by Betz et al.\ \cite{betz2025improved}. Moreover, two extreme cases are well understood, namely regular trees and the complete graph.

On the one hand, for regular trees, Angel \cite{angel_permutations} established the existence of two distinct phases for $\theta=1$ on $d$-regular trees with $d\geq 5$ and showed that infinite cycles appear for $\beta \in (d^{-1}+\tfrac{13}{6}d^{-2},\log(3))$ when $d$ is sufficiently large. Hammond \cite{hammond_infinitecycles} later proved that for $d\geq 3$, there exists a critical value $\beta_0$ above which infinite cycles emerge in the case $u=1$, a result extended by Hammond and Hegde \cite{hammond_infiniteloops} to all $u\in[0,1]$. Bj\"ornberg and Ueltschi provided an asymptotic expression for the critical parameter for loops up to second order in $d^{-1}$ \cite{bjornberg_ueltschi_trees1} and further extended this analysis to the case $\theta \neq 1$ \cite{bjornberg_ueltschi_treestheta}. A comprehensive picture for $d$-regular trees with $d\geq 3$, $\theta=1$, and $u\in[0,1]$ was established by Betz et al.\ in \cite{betz_sharptrees}, proving a (locally) sharp phase transition for infinite loops in the interval $(0,d^{-1/2})$ and deriving explicit bounds on the critical parameter up to order 5 in $d^{-1}$.

We also mention the related finite-volume problem on sparse random graphs, including random regular graphs, which are locally tree-like but globally expanding. In this setting the analogue of infinite loops is the occurrence of macroscopic loops or cycles; for random regular graphs this was proved in the interchange and quantum Heisenberg setting by Poudevigne \cite{poudevigne2022macroscopic}, and it was recently extended to the random loop model on sparse random graphs by Klippel \cite{klippel2026macroscopic}.

Here, the Galton--Watson setting is a natural instance of loop percolation under \emph{quenched disorder}, where the underlying tree is itself random. Betz et al.\ \cite{BetzEhlLees18} studied this case and established the existence of a phase transition under specific conditions on the offspring law. Theorems~A and~B show that the relation between the loop and link thresholds changes sharply between finite-mean offspring laws and sufficiently heavy-tailed infinite-mean laws.

The broader picture suggested by the known examples is the following.

\textbf{Problem.}\textit{ Determine whether, for every \(u\in[0,1]\), the critical values for loop percolation and link percolation on a tree \(T\) are different precisely when the boundary of the tree, defined in Section~\ref{sec:potential}, has finite positive dimension, possibly under natural regularity assumptions on that boundary.}

Schramm's result \cite{schramm_compositions_2005} on the complete graph is expected to generalise to a large number of dense graphs. Bj\"ornberg et al.\ \cite{interchange_with_reversals} showed the emergence of a $\operatorname{PD}(\tfrac12)$ structure on $K_n$ when $u\neq 1$. The agreement of the two critical parameters has also been established for the hypercube by Koteck\'y et al.\ \cite{kotecky_hypercube} and for Hamming graphs by Mi{\l}o\'s and \c{S}eng{\"u}l \cite{milos_sengul}.

Thus, trees represent sparse graphs whose structure becomes denser as their dimension increases, interpolating between the bounded-degree setting and the diverging-degree asymptotic regimes studied in previous work.

\paragraph{Outline of the remaining sections.}
Section~\ref{sec:model} defines the random loop model and states the two
Galton--Watson main theorems. Section~\ref{sec:coupling} constructs the local
coupling between loops and an auxiliary thinning of link percolation.
Section~\ref{sec:potential} introduces the boundary theory used in the proof,
states the deterministic criterion behind Theorem~A, proves the percolation
estimate behind it, and verifies the Galton--Watson input needed for
finite-mean offspring laws. Section~\ref{sec:weighted-pruning-proof} proves the
deterministic criterion, and
Section~\ref{sec:applications-proofs} proves Theorems~A and~B.

\section{Model and main results}\label{sec:model}
Precise definitions of the random loop model can be found, for example, in
\cite{ueltschi_random_loops} and \cite{betz_sharptrees}. We restrict
ourselves to a brief overview of the main objects. Let \( G = (V,E)\) denote a
simple, connected graph with vertex set \(V\) and edge set
\(E\subset\{\{x,y\}:x,y\in V, x\ne y\}\). In this paper, we always think of
\(G\) as a \emph{rooted graph}, i.e.\ a pair \((G,o)\) of a graph \(G\) and a
distinguished vertex \(o\in V\) that we call the \emph{root}.

The space underlying a configuration is \( E \times [0, 1) \), i.e.\ to each
edge, we attach a copy of the interval \( [0, 1) \). Let \( u \in [0,1]\) and
\(\beta>0 \) be fixed. On \( E \times [0, 1) \), we define two independent
homogeneous Poisson point processes: one with intensity \( \beta u \) whose
points we call \emph{crosses}, and another with intensity \( (1 - u)\beta \)
whose points we call \emph{(double) bars}. We refer to crosses and double bars
collectively as \emph{links}. The full configuration space is denoted by
\( \Omega(G) = \mathcal{N}^{E}\times \mathcal{N}^{E}\), where \(\mathcal{N}\)
is the space of simple counting measures on \(\mathcal{B}[0,1]\), the Borel
\(\sigma\)-field of the interval \([0,1)\). We denote the distribution of the
graph \(G\) edge-marked by the Poisson processes by \( \rho_{G,\beta,u} \).
When it is necessary to specify, we denote by \(I(e)\) the copy of \([0,1)\)
attached to \(e\). For a realization \( \omega \in \Omega(G) \), we denote by
\(m_e=m_e(\omega)\in \N\cup\{0\}\) the number of links on \(e\) and by
\(M=M(\omega)\in \N\cup\{0,\infty\}\) the total number of links.

The crosses and double bars form loops, defined by the following rules:

\begin{itemize}
    \item We start at a point \( (x,t)\in V\times[0,1) \) and move in the positive time direction, thereby exploring links on edges adjacent to $x$.
    \item If we encounter a double bar at $(\{x,y\}, \tau\pm)$, we jump to $(y,\tau\pm)$ and change time direction.
    \item If we encounter a cross at $(\{x,y\}, \tau\pm)$, we jump to $(y,\tau\mp)$ and continue in the same time direction.
    \item If, for any vertex $z$, we reach $(z,1)$ while moving in the positive time direction then we jump to $(z,0)$ and continue moving in the positive time direction. If we reach $(z,0)$ while moving in the negative time direction then we jump to $(z,1)$ and continue moving in the negative time direction.
    \item If we reach $(x,t)$ again, the loop is complete.
\end{itemize}

We observe that loops are closed trajectories in the space \( V \times [0, 1) \). This concept is best understood with the aid of a visual illustration such as in Figure \ref{fig: loop example}. Notice that the last rule means that we treat time as cyclic, in particular, we may interpret the $I(e),e\in E,$ as one-dimensional tori.

\begin{figure}
    \centering
\includegraphics[width=0.49\textwidth]{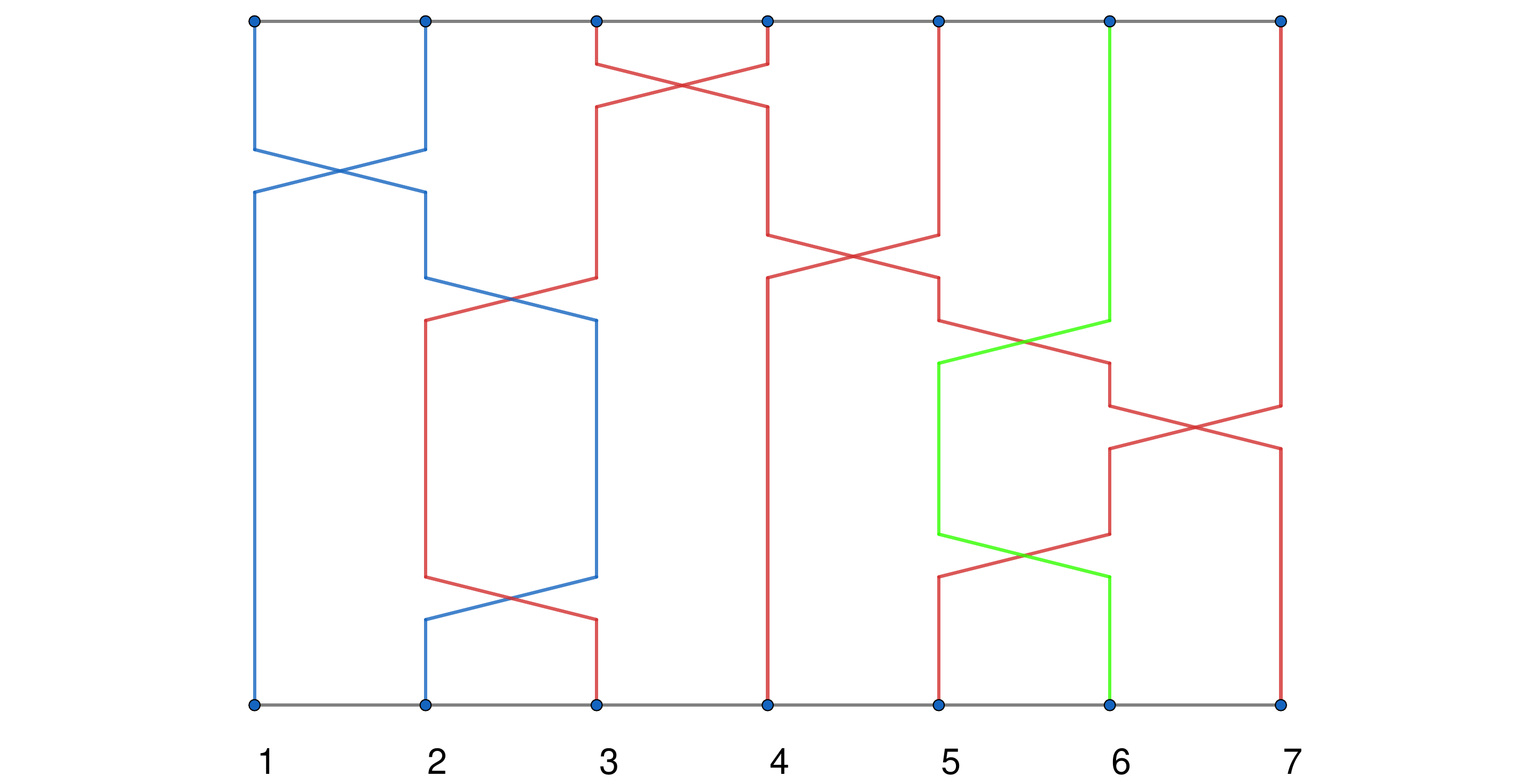}
\includegraphics[width=0.49\textwidth]{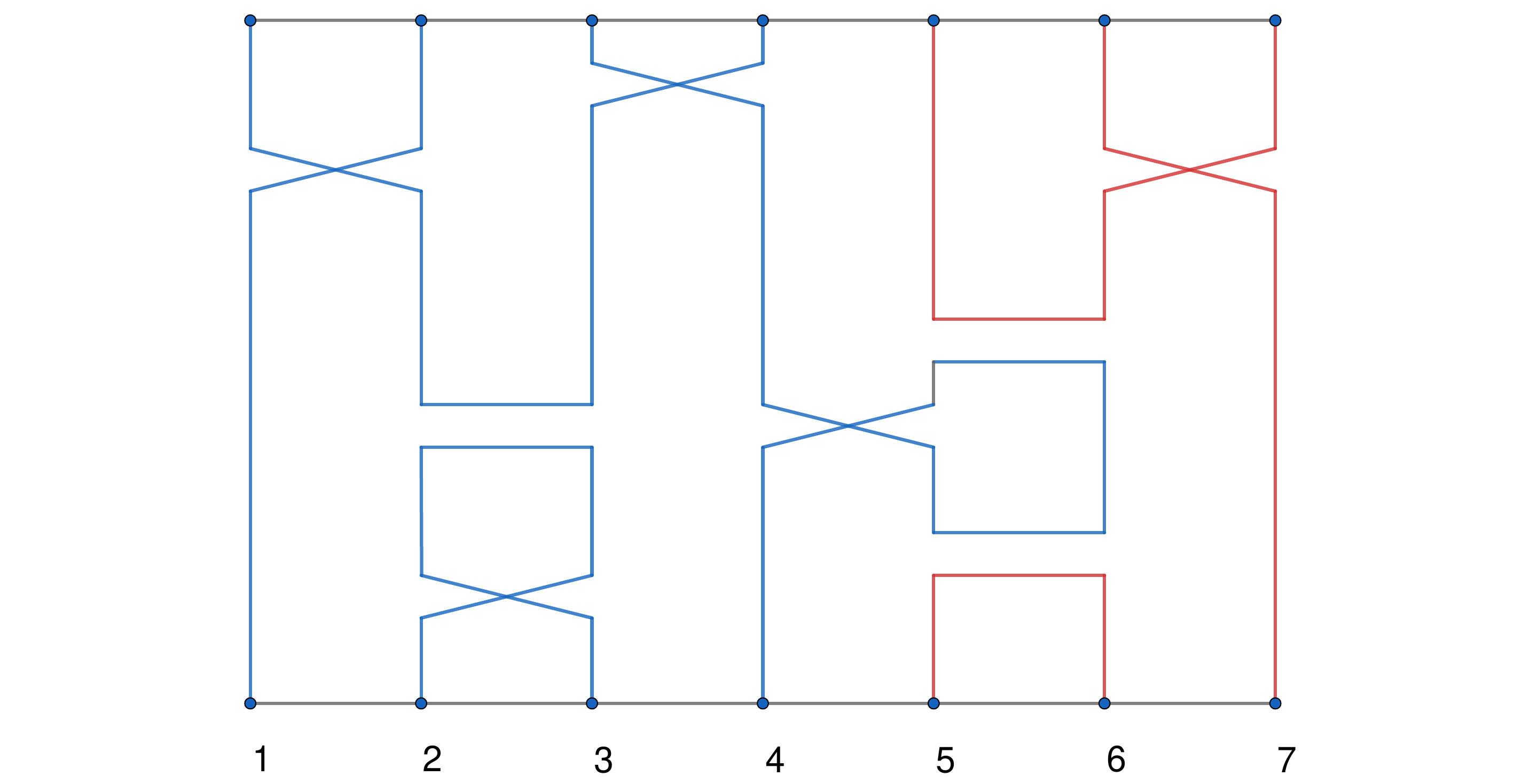}
\caption{Example loop configurations. On the left is an example of the case $u=1$ with only crosses. On the right is the case $u\in(0,1)$, where crosses and bars are both present.}
\label{fig: loop example}
\end{figure}

We denote the set of loops for a given configuration \( \omega \) by \( \mathcal{L}(\omega) \). On finite graphs, for \( \theta \in \mathbb{R}_{>0} \), the loop-weighted measure discussed in the introduction is defined by
\begin{equation*}
    \mathsf{loop}_{\beta,\theta}[G](\mathrm{d}\omega) \coloneqq \frac{1}{Z(G,\beta,u,\theta)} \, \theta^{\sharp \mathcal{L}(\omega)} \, \rho_{G,\beta,u}(\mathrm{d}\omega),
\end{equation*}
where \( Z(G,\beta,u,\theta) \) denotes the partition function, defined as
\begin{equation*}
    Z(G,\beta,u,\theta) \coloneqq \int_{\Omega(G)} \theta^{\sharp\mathcal{L}(\omega)} \,\rho_{G,\beta,u}(\mathrm{d}\omega)  .
\end{equation*}
When \( \theta \in \mathbb{N} \), this expression admits a natural interpretation after summing over colourings: each loop is uniformly assigned one of \( \theta \) distinct colours from the set \( [\theta] =\{1,\dots,\theta\}\).

We work in the case \( \theta = 1 \), for which the measure is simply \(\rho_{G,\beta,u}\) and is also defined on locally finite infinite graphs. Therefore, we simplify the notation to \( \mathsf{loop}_\beta[G] \).

Based on this framework, we define the concept of connection within the model. Two points \( (v,s) \) and \( (w,t) \) are said to be connected, denoted by \((v,s)\Leftrightarrow (w,t)\), if they are in the same loop. In particular, when considering the event that \( (v,0) \) is connected to the origin \( (o,0) \), we write this connection event as \( o \Leftrightarrow v \). Note that these conventions define \emph{two distinct} equivalence relations: one on $V\times[0,1)$ and one on $V$.

Let $L_o$ be the loop containing \( (o,0) \). By $|L_o|,$ we denote the total length of the loop, i.e.\ the total Lebesgue measure of the time the loop spends in each of its vertices. Note that jumps along links are considered instantaneous and are only represented as taking up positive time in the Figure~\ref{fig: loop example} for illustration purposes. Note that if $u=1$ and $|L_o|<\infty$ we have
\[
|L_o|=\sharp\{x:o\Leftrightarrow x\},
\]
i.e., the loop length corresponds precisely to the size of the permutation cycle induced by the crosses interpreted as transpositions of vertices along edges. We can define $L_x$ for other vertices $x\in V$ analogously. An \emph{infinite loop} is a loop of infinite length. Throughout, we use the convention \(\inf\emptyset=\infty\). We denote the critical parameter for loop percolation on $G$ or \emph{loop percolation threshold} of $G$ by
\begin{equation*}
    \beta^\textup{loop}_c(G) \coloneqq \inf\bigg\{\beta \geq 0 : \mathsf{loop}_\beta[G]\big(|L_o| = \infty\big) > 0 \bigg\}.
\end{equation*}
The infimum is well-defined by the convention \(\inf\emptyset=\infty\). The
loop parameter \(u\) is fixed throughout; when we need to emphasise this
dependence we write \(\beta^\textup{loop}_{c,u}(G)\). For connected \(G\), this
threshold does not depend on the choice of the root \( o \).
Indeed, if \(x\) and \(y\) are joined by a finite path, then by prescribing
links in small disjoint time intervals along the path and leaving the exterior
configuration unchanged, a positive-probability local event connects the loop
through \((x,0)\) to the loop through \((y,0)\). The same argument with \(x\) and
\(y\) interchanged transfers positive probability of an infinite root loop in
both directions. Note that $\beta^\textup{loop}_c(G)$ in general depends on $u$.

The loop measure induces another natural percolation structure where we obtain a random subgraph of $G$ by retaining edges with probability $1-e^{-\beta}$. This corresponds to \( \mathsf{loop}_\beta[G]\big(m_e \geq 1\big) \), recalling that \( m_e \) denotes the number of links on \( e \in E \). An edge is said to be \emph{retained} if there exists at least one link on it and it is said to be \emph{removed} otherwise. We denote the measure of this percolation model by \( \mathsf{link}_\beta[G] \) and refer to the corresponding percolation process as \emph{link percolation}.

We say that \( x \) and \( y \) are connected in the subgraph of $G$ obtained by $\mathsf{link}_\beta[G]$ if there exists a sequence of vertices \( (v_k)_{0 \leq k \leq K} \) such that \( v_0 = x \), \( v_K = y \), and \( \{v_k, v_{k+1}\}\in E \) is retained for every \( 0 \leq k < K \). This event is denoted by \( x \leftrightarrow y \). The connected component or \emph{cluster} of a vertex $x\in V$ with respect to link percolation on $G$ is defined as $C_x\coloneqq\{v\in V:x\leftrightarrow v\}$.

Furthermore, we write \( o \leftrightarrow \infty \) if the connected component of \( o \) under  \( \mathsf{link}_\beta[G] \) is infinite or, equivalently, there exists an infinite self-avoiding path of retained edges starting at $o$. The critical parameter for the link percolation model is defined as
\begin{equation*}
    \beta^\textup{link}_c(G) \coloneqq \inf\big\{\beta \geq 0 \,:\, \mathsf{link}_\beta[G]\big(o \leftrightarrow \infty\big) > 0\big\}.
\end{equation*}
Unlike $\beta^\textup{loop}_c(G)$, $\beta^\textup{link}_c(G)$ never depends on $u$.
\begin{remark}
Since $\rho_{G,\beta,u}$ is derived from an i.i.d.\ collection of Poisson
processes, the event that there exists an infinite link cluster is a tail event.
The analogous event for loops is also unchanged by altering finitely many edge
configurations: finite modifications can only split or merge finitely many loop
pieces, and hence cannot eliminate all infinite loops or create one from only
finite loops. Kolmogorov's $0$-$1$ law therefore implies that
\(\mathsf{link}_\beta[G](\exists x\in V: x \leftrightarrow \infty)\) and
\(\mathsf{loop}_\beta[G](\exists x\in V: |L_x|=\infty)\) are each either \(0\) or
\(1\). Together with countability and the root-independence observation above,
this shows that the positivity condition at \(o\) is equivalent to the
corresponding almost-sure existence condition somewhere in the graph. Thus we
could equivalently have defined the critical value(s) using the latter
expression(s).
\end{remark}

We now state our main Galton--Watson separation result.

\begin{theoremA}\label{thm:gw-finite-mean-main}
Fix \(u\in[0,1]\). Let \(T\) be a Galton--Watson tree with offspring law
\(\zeta=(\zeta_0,\zeta_1,\ldots)\), offspring variable \(Z\), \(Z<\infty\)
almost surely, and mean \(m=\E Z\in(1,\infty)\). Then, conditioned on survival,
\[
\beta_c^\mathsf{loop}(T)>\beta_c^\mathsf{link}(T)
=-\log(1-m^{-1})
\]
almost surely.
\end{theoremA}

Theorem~A follows from the stronger deterministic weighted pruning criterion,
Theorem~\ref{thm:weighted-pruning-criterion}, stated in
Section~\ref{sec:potential}. That criterion is more technical because it is
phrased in terms of boundary rays and the effect of a local pruning scheme
along those rays.

The next theorem is a partial converse for sufficiently heavy-tailed
Galton--Watson trees in the crosses-only model. Since the conclusion is an
annealed statement, write \(\mathbb E_\zeta\) for expectation over a
Galton--Watson tree \(T\) with offspring law \(\zeta\). We define the
\emph{annealed critical values} by
\[
\beta_{c,\mathrm{ann}}^\mathsf{loop}(\zeta)
\coloneqq
\inf\Big\{\beta\geq0:
\mathbb E_\zeta\big[\mathsf{loop}_\beta[T](|L_o|=\infty)\big]>0
\Big\},
\]
and
\[
\beta_{c,\mathrm{ann}}^\mathsf{link}(\zeta)
\coloneqq
\inf\Big\{\beta\geq0:
\mathbb E_\zeta\big[\mathsf{link}_\beta[T](o\leftrightarrow\infty)\big]>0
\Big\}.
\]
Thus the positivity condition is imposed after averaging over both the
Galton--Watson tree and the link or loop configuration, respectively.

\begin{theoremB}\label{thm:gw-heavy-tail-main}
Consider the crosses-only model \(u=1\). Let \(T\) be a Galton--Watson tree
with offspring law \(\zeta=(\zeta_0,\zeta_1,\ldots)\), offspring variable
\(Z\), \(Z<\infty\) almost surely, and probability generating function
\[
f(z)=\E z^Z,\qquad z\in[0,1].
\]
If
\[
\liminf_{\varepsilon\downarrow0}
\sqrt{\varepsilon}\,f'(1-\varepsilon)>\frac1{\sqrt2},
\]
then
\[
\beta_{c,\mathrm{ann}}^\mathsf{loop}(\zeta)
=
\beta_{c,\mathrm{ann}}^\mathsf{link}(\zeta)
=0.
\]
In particular, the displayed tail condition holds if, up to normalisation,
\(\zeta_k=k^{-\tau}L(k)\) for \(k\geq1\), where \(L\) is slowly varying, meaning that
\(L(cx)/L(x)\to1\) as \(x\to\infty\) for every \(c>0\), and
\(\tau\in(1,\frac32)\).
\end{theoremB}

If \(m=\E Z\leq1\), then the equality
\(\beta_c^\mathsf{loop}(T)=\beta_c^\mathsf{link}(T)=\infty\) holds for almost
every realised tree \(T\). The proofs of Theorems~A and~B are given in
Section~\ref{sec:applications-proofs}. The stronger deterministic theorem
implying Theorem~A is stated later in Section~\ref{sec:potential}.

We close this section with a quantitative local calculation illustrating why
local changes that are irrelevant for link percolation can still affect loop
exploration.

\begin{example}
Let \(T_d\) be the rooted \(d\)-ary tree, so every vertex has \(d\) children.
Then \(1-\exp(-\beta_c^\mathsf{link}(T_d))=1/d\). On this unmodified tree,
the results of Betz et al.\ \cite[Equation (2.3)]{betz_sharptrees} give, in
particular, \(\beta_c^\mathsf{loop}(T_d)=d^{-1}+O(d^{-2})\).

Now form \(T_d(S)\) by adding \(S\) pendant leaves to every vertex of \(T_d\).
This operation does not change the link critical value: an infinite link
cluster in \(T_d(S)\) exists if and only if the restriction of the link
configuration to the original copy of \(T_d\) contains an infinite connected
component.

The loop process is more sensitive. In the crosses-only case \(u=1\), a pendant
leaf with at least two links to its parent can send the trajectory from a
vertex \(x\) into the leaf and then return it to \(x\) at a later time before
the trajectory uses another edge of the original tree. For large \(S\), such
leaf excursions occur on a much shorter time scale than links on the original
tree and therefore act as a local time-randomisation mechanism.

The following calculation records the scale of this effect. Suppose the
exploration has entered a non-root vertex \(x\) from its parent. Let \(M\) be
the number of links on the parent edge of \(x\), and let \(R\) be the total
number of links on the \(d\) child edges of \(x\) in the original tree. Before
conditioning, \(M\) and \(R\) are independent Poisson random variables with
means \(\beta\) and \(\beta d\), respectively. Conditional on \(M\geq1\) and
\(R\geq1\), the idealised local picture in which the pendant-leaf excursions
fully resample the return time gives the backtracking factor
\begin{equation*}
\begin{aligned}
A_d(\beta)
&\coloneqq
\frac{1}{(1-e^{-\beta})(1-e^{-\beta d})}
\sum_{m,r\geq1}
\bigg(\frac{m}{m+r}\bigg)^m
\mathbb{P}(M=m)\mathbb{P}(R=r)  \\
&\geq
\frac{1}{(1-e^{-\beta})(1-e^{-\beta d})}
\sum_{m,r\geq1}
e^{-r}\mathbb{P}(M=m)\mathbb{P}(R=r) \\
&=
\frac{e^{-\beta d(1-e^{-1})}-e^{-\beta d}}{1-e^{-\beta d}}.
\end{aligned}
\end{equation*}
Indeed, the inequality uses
\((m/(m+r))^m=(1+r/m)^{-m}\geq e^{-r}\). Consequently
\[
\liminf_{\beta\downarrow0} A_d(\beta)\geq e^{-1}.
\]
While this is only a local backtracking calculation, it nevertheless shows
quantitatively why the link cluster alone does not determine the effective
transition probabilities seen by loop exploration, and hence why simple
monotonicity statements for loop critical values under local tree
modifications are delicate.
\end{example}

\section{Coupling to inhomogeneous Bernoulli percolation}\label{sec:coupling}

From now on, we consider only rooted trees $(T,o)$, with $T=(V,E)$ and $o\in V$. Fix $x\in V\setminus\{o\}$. We write $|x|$ for the graph distance in $T$ of $x\in V$ to the root $o$; we write $a(x)$ for the \emph{parent} of $x$, i.e.\ the unique neighbour $y$ of $x$ with $|y|=|x|-1$; and we write $e_x=\{x,a(x)\}$ for the edge connecting $x$ to its parent. The latter definition induces a bijection $V\setminus\{o\}\to E$ and we denote the corresponding inverse by $x:E\to V\setminus\{o\}$, i.e.\ $x(e)$ is the vertex of $e$ further from the root. We also frequently suppress the edge set and the vertex set in the notation and just write $x\in T$ for vertices $x$ or $e\in T$ for edges $e$.

Our method to prove that the critical values differ relies on a coupling of $\mathsf{loop}_\beta[T]$ to variants of simplified percolation models. We keep the loop parameter \(u\in[0,1]\) fixed and separate the local deletion probabilities for \(u>0\) and for the bar-only case \(u=0\).

\begin{definition}
Fix \(u\in[0,1]\). For admissible \(d\geq2\), \(d^\ast\geq1\), and
\(\lambda>0\), let \(R_u(d,d^\ast,\lambda)\) be given, for \(u\in(0,1]\), by
\begin{equation}\label{eq:blockingprobabilities}
R_u(d,d^\ast,\lambda)=\frac{\lambda^3 \textup{e}^{-2\lambda}u^2}{2(1-\textup{e}^{-\lambda})^2}\int_{0}^1\int_0^1\Big( \frac{\textup{e}^{-\lambda}(\textup{e}^{\lambda (1-|s-t|)}-1)}{1-\textup{e}^{-\lambda}} \Big)^{d^\ast -1}\Big( \frac{\textup{e}^{-\lambda}(\textup{e}^{\lambda |s-t|}-1)}{1-\textup{e}^{-\lambda}}\Big)^{d -2}\,\textup{d}s\,\textup{d}t,
\end{equation}
and, for \(u=0\), by
\begin{equation}\label{eq:blockingprobabilities2}
R_0(d,d^\ast,\lambda)=\frac{\lambda^3 \textup{e}^{-2\lambda}}{2(1-\textup{e}^{-\lambda})^2}\int_{0}^1\int_0^1\Big( \frac{\textup{e}^{-\lambda}(\textup{e}^{\lambda |s-t|}-1)}{1-\textup{e}^{-\lambda}} \Big)^{d+d^\ast -3}\,\textup{d}s\,\textup{d}t.
\end{equation}
Let $(T,o)$ be a fixed rooted tree and define, for a vertex with at least one child,
$$D^\ast(x)=D_T^\ast(x)=\min\{\deg_T(y): a(y)=x\}.$$
Vertices without children are removed with probability \(0\). Remove vertices
\(x\in T\setminus\{o\}\) with at least one child independently with probability
\(R_u(\deg_T(x),D^\ast(x),\lambda)\). We call the resulting independent site
percolation model on $(T,o)$ \emph{pruning percolation}, denoted by
$\mathsf{prune}_\lambda[T]$. A related model, which we dub \emph{delayed
pruning percolation} and denote by $\mathsf{delay}_\lambda[T]$ is obtained by
replacing these deletion probabilities by
\[
\1_{\{|x|\equiv1\!\!\!\pmod 3\}}R_u(\deg_T(x),D^\ast(x),\lambda),
\]
for vertices \(x\) with at least one child, and by deletion probability \(0\)
for vertices without children.
\end{definition}
The formulas \eqref{eq:blockingprobabilities} and
\eqref{eq:blockingprobabilities2} are derived from the local pruning-edge
calculation in the proof of Lemma~\ref{lem:coupling1}.
Pruning percolation allows any non-root vertex with at least one child to be
deleted, while delayed pruning only allows such deletions at levels congruent
to \(1\) modulo \(3\). We mainly use the delayed variant because nearby local loop
events need not be independent. The level spacing makes the edge sets used by
the selected local pruning events disjoint, as checked in the proof of
Lemma~\ref{lem:coupling1}. This coupling refines the idea employed in
\cite{muehlbacher_critical_2021}.

\begin{definition}\label{def:blockingedges}
    For two times \(s,t\in[0,1)\), set
    \(H(s,t)=(s\wedge t,s\vee t)\) and
    \(J(s,t)=[0,1)\setminus [s\wedge t,s\vee t]\). Endpoints carry the two
    links under discussion and are not part of the exposed intervals. Consider a
    realisation $\omega$ of $\mathsf{loop}_\beta[T]$. For a point
    $(z,t)\in V\times[0,1)$ begin an exploration of the loop structure by moving
    in the positive time direction. If \(u>0\), we call an edge
    $e=\{y,x(e)\}$ encountered during this exploration \emph{pruning for}
    $\omega$ \emph{and} $(z,t)$ if we encounter precisely two crosses on $e$, at
    times \(s_1,s_2\in I(e)\), and while traversing the \emph{exposed
    intervals} \(x(e)\times J(s_1,s_2)\) and \(y\times H(s_1,s_2)\) we do not
    encounter any links. If \(u=0\), all links are bars, and we call \(e\)
    pruning if the exploration encounters precisely two bars on \(e\), at times
    \(s_1,s_2\), and the exposed intervals \(x(e)\times J(s_1,s_2)\) and
    \(y\times J(s_1,s_2)\) contain no links. If $e$ is pruning for $\omega$ and
    $(o,0)$, we simply call it \emph{pruning}.
\end{definition}
The times \(s_1\) and \(s_2\) are defined by the exploration of the loop; the
notation above orders them only to specify the two complementary time intervals.
In particular, \(J(s_1,s_2)\) is the complementary arc on the time circle and is
written as two intervals after cutting the circle at \(0\). Suppose that
\(x\ne o\), that \(a(x)\) is the parent of \(x\), and that the exploration enters
\(x\) through the parent edge \(\{a(x),x\}\). If a child edge \(\{x,y\}\) is
pruning, then this visit crosses to \(y\), is forced by the exposed interval at
\(y\) to cross back to \(x\), and is then forced by the exposed interval at
\(x\) to return to the parent edge before entering another descendant edge. In
the coupling below this local event is used together with the event that the
parent edge of \(x\) carries precisely one link. Hence a later return from the
parent side would have to use the same parent link and would return to the
already exposed local segment. Lemma~\ref{lem:coupling1} makes this connection
precise: the descendant subtree below such an \(x\) is pruned with probability
determined by \(R_u\), except for a finite local remnant that cannot contain an
infinite descendant branch.
\begin{remark}
    The important difference between `blocking' edges as used in \cite{muehlbacher_critical_2021} and pruning edges is that a blocking edge just prevents loops from crossing it, whereas a pruning edge cuts off infinite descendant branches from \(L_o\), up to the finite remnants discussed above; see Figure~\ref{fig:prune} below. Although, in the tree case, the effect on the overall geometry is comparable, our construction has the advantage of preserving more independence between adjacent pruning events at vertices having the same parent than using blocking edges. Also note that we include $u=0$, whereas \cite{muehlbacher_critical_2021} does not.
\end{remark}

\begin{figure}
    \centering
 \includegraphics[height=6cm,width=0.32\textwidth]{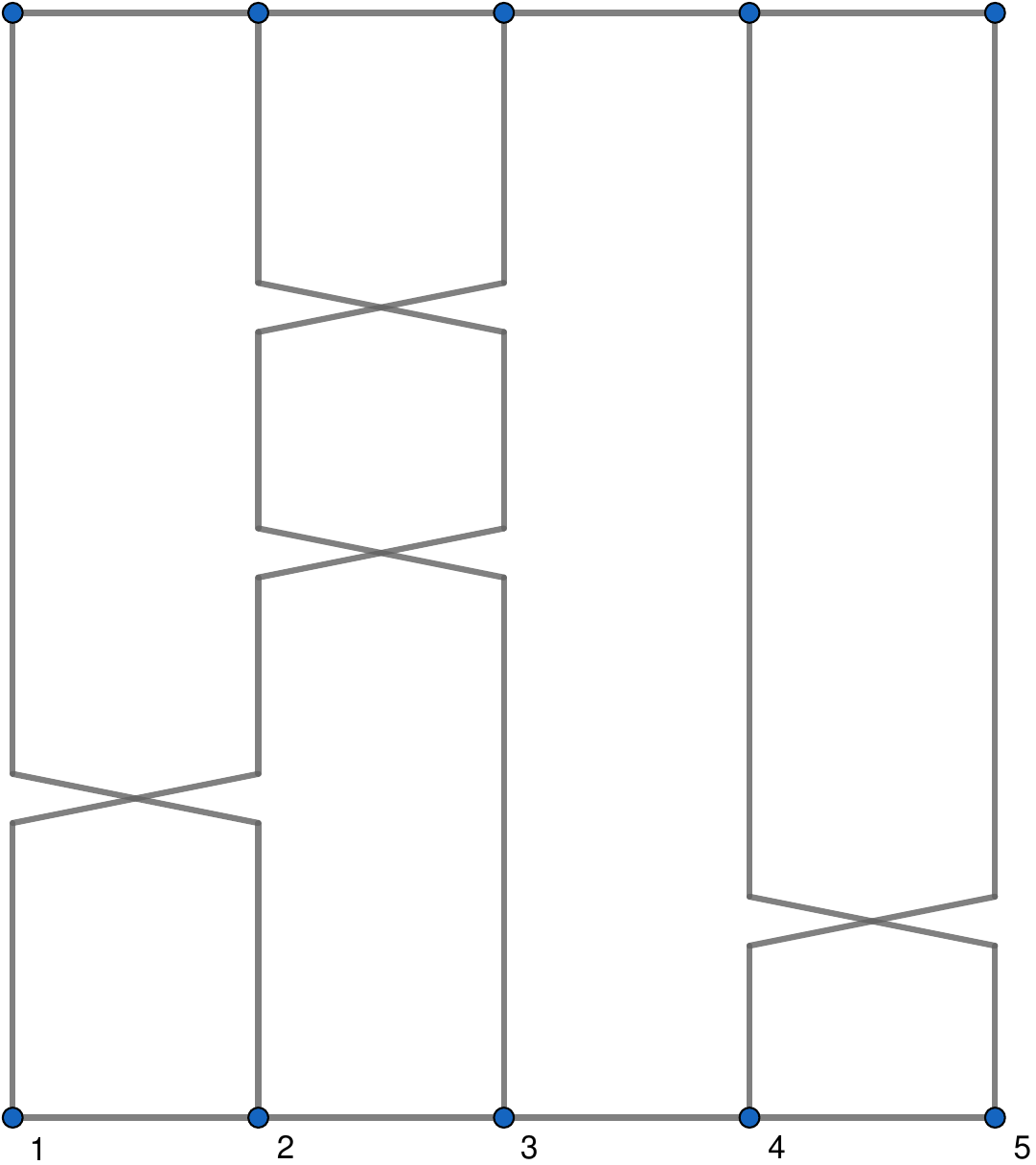}
\includegraphics[height=6cm,width=0.32\textwidth]{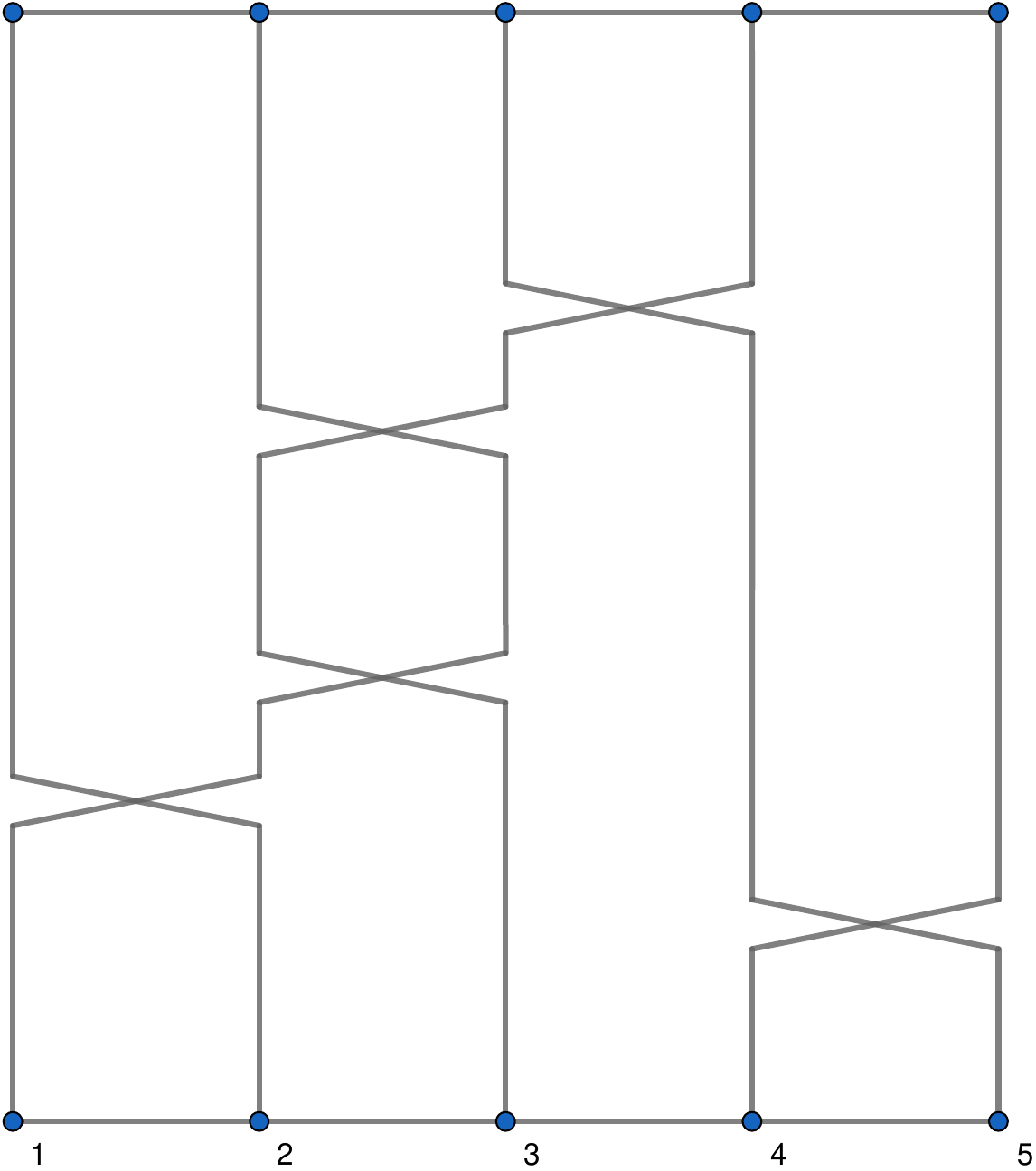}
\includegraphics[height=6cm,width=0.32\textwidth]{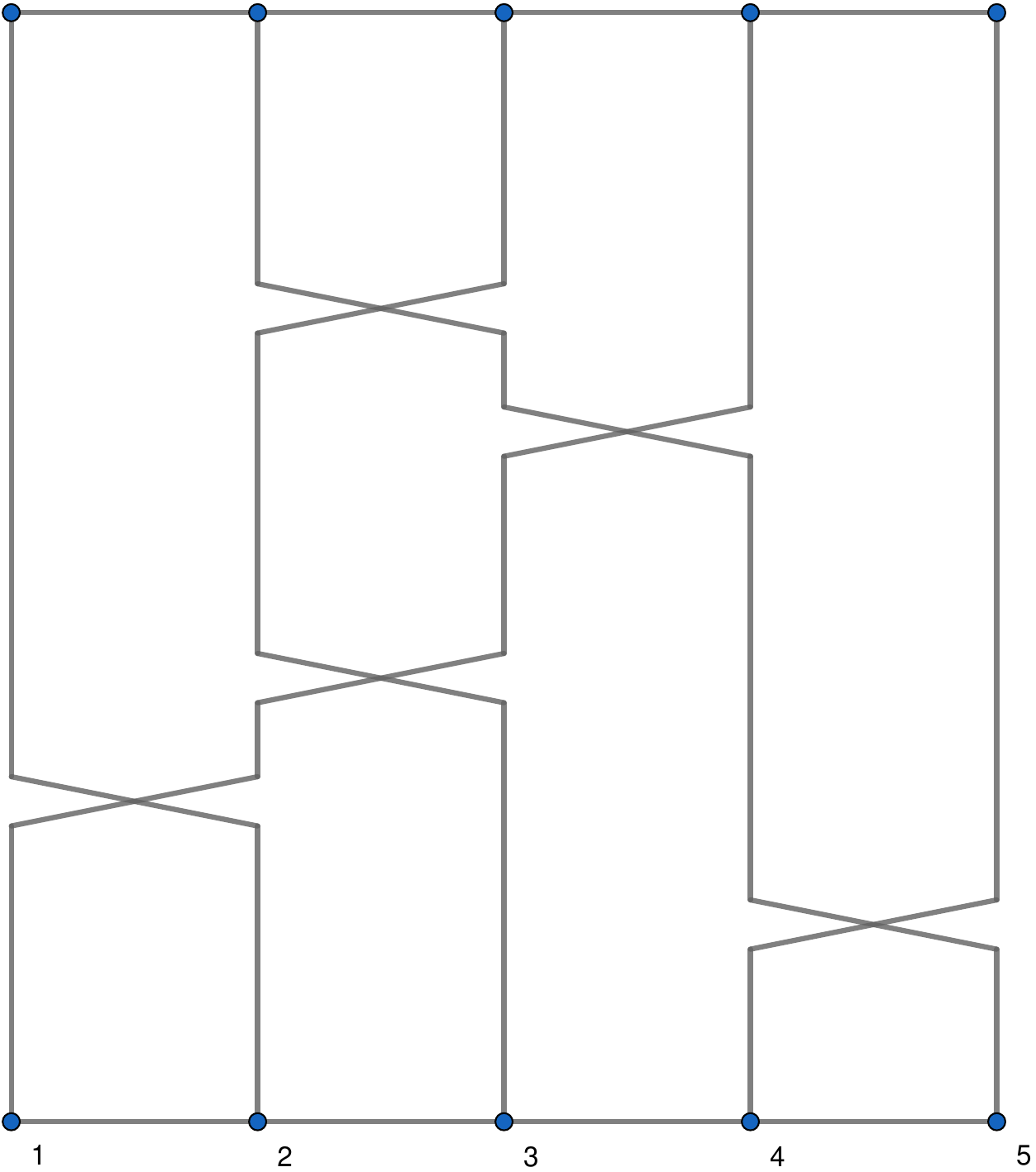}
\includegraphics[width=0.8\textwidth]{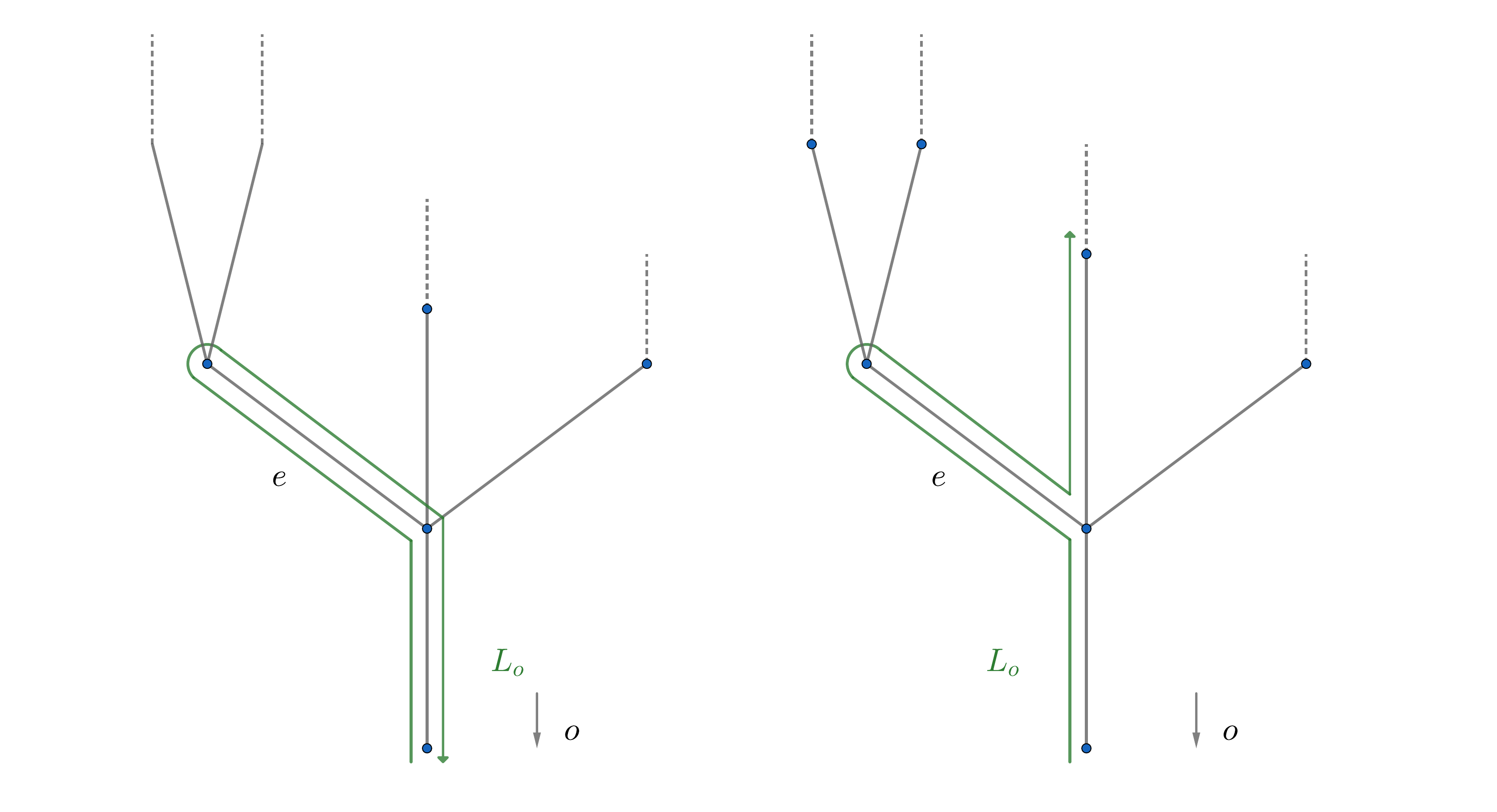}
\caption{Top: examples of an edge ($\{2,3\}$) and three link configurations so that the edge is pruning, blocking but not pruning, and neither pruning nor blocking, respectively. Bottom left: the effect of a pruning edge $e$ on the path of $L_o$ in $T$. Bottom right: the effect of a blocking edge $e$ on the path of $L_o$ in $T$.}
\label{fig:prune}
\end{figure}

To formalize the coupling between pruning edges and pruning percolation, we use
only the comparison needed for percolation. For two random rooted subgraphs of
the same base tree, write \(\mathsf A\preceq_\infty\mathsf B\) if they can be
coupled so that every infinite branch in the root component of \(\mathsf A\)
forces an infinite branch in the root component of \(\mathsf B\). When
\(\mathsf A=\mathsf{loop}_\beta[T]\), the root component means the set of
vertices visited by \(L_o\). Finite local remnants left by pruning events are
irrelevant for this comparison because they cannot contain an infinite branch.

Additionally, for percolation processes $\mathsf{A},\mathsf{B}$ on a base graph $G$, we write $\mathsf{B}\circ \mathsf{A}[G]$ for the distribution of a percolation configuration obtained by first percolating $G$ with $\mathsf{A}$ and then using the obtained subgraph of $G$ as input for the percolation process $\mathsf{B}$.
\begin{lemma}\label{lem:coupling1}
For any tree, we have
\[
\mathsf{loop}_\beta[T]\preceq_\infty
\mathsf{delay}_\beta\circ \mathsf{link}_\beta[T]\preceq_\infty
\mathsf{link}_\beta[T].
\]
Here the pruning probabilities in $\mathsf{delay}_\beta$ are computed in the
link cluster obtained in the first step.
\end{lemma}
\begin{proof}
The second relation is immediate, since delayed pruning only removes vertices
from the link cluster. We prove the first relation directly, without comparing
the two possible orders of link and pruning percolation.

For any edge $e\in E$, let $A_1(e)$ denote the event that $e$ carries precisely
one link. We now work conditionally on the link-cluster $C_o$ of $o$ in $T$.
Fix any edge $e\in E(C_o)$ with endpoint $x(e)$ and let $N$ denote the set of
offspring of $x(e)$ in $C_o$, i.e.\ the vertices neighbouring $x(e)$ that are
further away from the root than $x(e)$. Assume for the moment that \(N\) is
nonempty. Let \(y^\ast\) denote a fixed vertex in \(N\), and denote its degree
by \(D^\ast\). Let
\[
B(e,y^\ast)=\big\{\{x(e),y^\ast\}\text{ is pruning}\big\}.
\]
For \(s\in[0,1]\), set
\[
        Q_\beta(s)
        =
        \frac{\textup{e}^{-\beta}}{1-\textup{e}^{-\beta}}
        \sum_{k=1}^\infty \frac{\beta^k s^k}{k!}
        =
        \frac{\textup{e}^{-\beta}(\textup{e}^{\beta s}-1)}
             {1-\textup{e}^{-\beta}}.
\]
First suppose \(u>0\). Then, for \(D=\textup{deg}_{C_o}(x(e))\geq 2\), and for independent
\(\operatorname{Uniform}[0,1)\) variables \(U_1,U_2\),
\begin{align*}
&\mathsf{loop}_\beta[T]\big(A_1(e)\cap B(e,y^\ast)\mid C_o\big)\\
&\quad =
   \frac{\beta \textup{e}^{-\beta}}{1-\textup{e}^{-\beta}}\,
   \frac{\beta^2\textup{e}^{-\beta}u^2}{2(1-\textup{e}^{-\beta})}\,
   \E\Big[
        Q_\beta\big(U_1\wedge U_2+1-U_1\vee U_2\big)^{D^\ast-1}
        Q_\beta\big(U_1\vee U_2-U_1\wedge U_2\big)^{D-2}
      \Big].
\end{align*}
The two prefactors are the conditional probabilities that \(e\) carries
precisely one link and that \(\{x(e),y^\ast\}\) carries precisely two links
(both of them crosses). These events are independent of each other and of the link status of the
remaining edges adjacent to \(x(e)\) and \(y^\ast\). The two factors in the
expectation correspond to the events that links on the remaining adjacent edges
fall into the complements of the exposed intervals. Conditionally on the
exposed intervals, these link events are all independent. Finally, on the event
\(A_1(e)\), we may calculate the positions of the exposed intervals relative
to the position of the single link on \(e\), which gives the two complementary
lengths displayed above.

Gathering terms and setting $\ell= U_1\wedge U_2+1-U_1\vee U_2$ yields
\begin{equation*}
\mathsf{loop}_\beta[T]\big(A_1(e)\cap B(e,y^\ast)\mid C_o\big)
=
\frac{\beta^3 \textup{e}^{-2\beta}u^2}{2(1-\textup{e}^{-\beta})^2}
\E\Big[
    Q_\beta(\ell)^{D^\ast-1} Q_\beta(1-\ell)^{D-2}
  \Big],
\end{equation*}
which is \(R_u(D,D^\ast,\beta)\), the deletion probability used by
\(\mathsf{delay}_\beta[C_o]\) at \(x(e)\).

If \(u=0\), the same conditioning is used but the two links on
\(\{x(e),y^\ast\}\) are bars. Crossing a bar reverses the vertical direction,
so after the first crossing of the child edge the exploration traverses the
complementary interval until it reaches the second bar, and after crossing back
it again traverses the complementary interval at \(x(e)\) before returning to
the parent edge. Thus the exposed intervals on both sides have length \(\ell\),
and all links on the remaining adjacent edges must fall in the complementary
allowed interval of length \(1-\ell=U_1\vee U_2-U_1\wedge U_2\). These edges
therefore contribute
\[
Q_\beta(1-\ell)^{D+D^\ast-3}.
\]
This gives exactly \(R_0(D,D^\ast,\beta)\), the bar-only probability in
\eqref{eq:blockingprobabilities2}.

Finally, for every vertex \(x\in C_o\) with \(|x|\equiv1\pmod3\) and at least
one child in \(C_o\), choose a child \(y^\ast(x)\) of minimal degree in
\(C_o\), using independent auxiliary marks to break ties. The event attached to
\(x\) uses only the parent edge \(e_x\), the chosen child edge
\(\{x,y^\ast(x)\}\), and the other edges incident to \(x\) and to
\(y^\ast(x)\); it does not use the other edges incident to the parent \(a(x)\).
If \(x\) and \(z\) are two distinct selected vertices, then these edge sets are
disjoint. Indeed, if one selected vertex is an ancestor of the other, their
levels differ by at least \(3\), so the upper event reaches at most to edges
between levels \(|x|+1\) and \(|x|+2\), while the lower event starts with its
parent edge. If neither selected vertex is an ancestor of the other, the two
local edge sets lie in disjoint branches; sibling vertices cause no overlap
because their events do not use the unused edges incident to their common
parent. Conditional on \(C_o\) and on the auxiliary marks, we have only
conditioned each edge on either carrying no link or carrying at least one link,
so the Poisson configurations on disjoint edge sets remain independent.
Therefore the corresponding pruning events can be coupled with the independent
deletions in \(\mathsf{delay}_\beta[C_o]\). Under this coupling, if a selected
vertex \(x\) is removed, the event \(A_1(e_x)\) and the pruning child edge force
the exploration from the parent side back to the unique parent-edge link before
any other child edge of \(x\) can be used. Any later return from the parent side
uses that same link and returns to the already exposed local segment.
Consequently, below such a removed \(x\), the loop contributes only a finite
local remnant; every infinite descendant branch through \(x\) is cut off.

\end{proof}
We close this section by stating some elementary properties of the local pruning
probability \(R_u\).
\begin{prop}\label{prop:obvious}
Fix $u\in[0,1]$. For admissible pairs \(d\geq2\), \(d^\ast\geq1\),
the function \(R_u(d,d^\ast,\lambda)\) is non-increasing in \(d\) and
\(d^\ast\) and continuous in \(\lambda\in(0,\infty)\). We furthermore have
\[
R_u(d,d^\ast,\lambda)>0 \quad \text{ for all }d\geq2,\ d^\ast\geq1,\ \lambda\in (0,\infty),
\]
and
\[
\lim_{\lambda\to 0}R_u(d,d^\ast,\lambda)=\lim_{\lambda\to \infty}R_u(d,d^\ast,\lambda)=\lim_{d\to \infty}R_u(d,d^\ast,\lambda)=\lim_{d^\ast\to \infty}R_u(d,d^\ast,\lambda)=0.
\]
\end{prop}
\begin{proof}
Write
\[
        Q_\lambda(s)=
        \frac{\textup{e}^{-\lambda}(\textup{e}^{\lambda s}-1)}
             {1-\textup{e}^{-\lambda}}
        =
        \frac{\textup{e}^{\lambda s}-1}{\textup{e}^{\lambda}-1},
        \qquad 0\leq s\leq 1 .
\]
Then \(0\leq Q_\lambda(s)\leq1\), with \(Q_\lambda(s)<1\) for \(s<1\) and
\(Q_\lambda(s)>0\) for \(s>0\). Increasing \(d\) or \(d^\ast\) only increases
one of the exponents of a factor lying in \([0,1]\), so \(R_u\) is
non-increasing in both degree variables.

For fixed admissible \(d,d^\ast\), the prefactor is continuous in
\(\lambda>0\), and the integrand is bounded by \(1\) and continuous in
\(\lambda\). Dominated convergence gives continuity. Positivity follows
because the prefactor is positive and the integrand is positive on the
positive-measure set \(\{(s,t):0<|s-t|<1\}\).

The limits in \(d\) and \(d^\ast\) follow again from dominated convergence:
as the relevant exponent tends to infinity, the integrand tends to \(0\) for
Lebesgue-a.e. \((s,t)\). Finally,
\[
        \frac{\lambda^3\textup{e}^{-2\lambda}}
             {2(1-\textup{e}^{-\lambda})^2}
        \longrightarrow 0
        \quad\text{as }\lambda\downarrow0
        \quad\text{and as }\lambda\uparrow\infty,
\]
while the integrals are bounded by \(1\). This proves the two limits in
\(\lambda\).
\end{proof}

\section{Potential theory and percolation on trees}\label{sec:potential}
We call an infinite self-avoiding path \(\xi=(o,\xi_1,\xi_2,\dots)\) in \(T\) a \emph{ray}. The set of rays is the \emph{boundary} \(\partial T\) of \(T\). This agrees with the usual end-boundary, since two rays starting at the root of a tree represent the same end only when they are identical. For \(\xi,\xi'\in\partial T\), define \(\xi\wedge \xi'\) to be the vertex furthest away from the root that is common to both \(\xi\) and \(\xi'\). Then \((\xi,\xi')\mapsto\textup{e}^{-|\xi\wedge\xi'|}\) defines a metric on \(\partial T\) with respect to which the boundary forms a complete separable metric space. As in the introduction, we write \(\dim(X)\) for the Hausdorff dimension of a metric space \(X\).

To study more general trees than Galton--Watson trees, we use harmonic analysis
on trees and its connection with percolation, developed in
\cite{lyons_percolation_1990,lyons_capacity_1992,lyons_ergodicGW_1995}. Our
presentation relies mostly on the monograph \cite{lyons_probability_2017}.
Consider a rooted infinite tree \((T,o)\). A function
\(g:T\to[0,\infty)\) on the vertices of \(T\) is called a \emph{gauge} if it
is non-decreasing along rays. By monotonicity, we may extend any gauge to a
function \(\partial T\to [0,\infty]\), which we also denote by \(g\). A
gauge \(g\) gives rise to a \emph{kernel}
\(\mathscr{k}_g:\partial T\times\partial T\to [0,\infty]\) via
\[
\mathscr{k}_g(\xi,\xi')=g(\xi\wedge\xi').
\]
By $\Delta g$ we denote the function $g(x)-g(a(x))\geq 0, x\neq o$, i.e.\ the gauge difference between \(x\) and its parent (not to be confused with the graph Laplacian). We set \(\Delta g(o)=g(o)\) for the root term in the potential sum below. We may think of \(\Delta g\) as a cost or resistance encountered when traversing the edge \(\{a(x),x\}\) from the root. This interpretation can be made rigorous using the electrical network formalism; see \cite{lyons_probability_2017}. Consequently, \(\Delta g\) can also be interpreted as a function of the edges in \(T\), which is occasionally convenient. Recall that the boundary \(\partial T\) of \(T\) carries a metric structure, so we can define
\[
\Pi(\partial T)=\{\nu : \nu \text{ is a Borel measure on }\partial T \text{ with }\nu(\partial T)=1\}.
\]
The \emph{potential} of $\nu\in\Pi(\partial T)$ with respect to the gauge $g$ is defined as
\[
\mathscr{h}_{g,\nu}(\xi)=\int_{\partial T} \mathscr{k}_g(\xi,\xi')\nu(\textup{d}\xi'),\quad \xi\in \partial T,
\]
and the corresponding \emph{energy} is
\[
\mathscr{E}_g(\nu)=\int_{\partial T\times \partial T} \mathscr{k}_g(\xi,\xi')\nu\otimes\nu(\textup{d}\xi,\textup{d}\xi')=\int_{\partial T} \mathscr{h}_{g,\nu}(\xi)\nu(\textup{d}\xi).
\]
The following alternative representation of the potential was established in
\cite{lyons_percolation_1990}; see also
\cite[Prop. 16.1]{lyons_probability_2017}:
\begin{equation}\label{eq:potentialsum}
\mathscr{h}_{g,\nu}(\xi)=\sum_{x\in \xi}\Delta g(x)\nu(b_x),
\end{equation}
where the ball $b_x\subset \partial T$ around $x\in T$ consists of those rays $\xi\in \partial T$ that contain $x$.

Finally, the \emph{\(g\)-capacity} of a Borel set $A\subset \partial T$ is the quantity
\[
\mathscr{C}_g(A)=\inf\{ \mathscr{E}_g(\nu): \nu\in\Pi(\partial T), \,\nu(A)=1 \}^{-1}.
\]
The connection of these potential theoretic notions with independent Bernoulli percolation is as follows: any independent Bernoulli percolation on $T$ can be parametrised through its edge retention probabilities
\[
p(e), e\in T,
\]
which gives rise to a gauge function $g_p$ via $g_p(o)=1$ and
\[
g_p(x)=\prod_{o<v\leq x}p(e(v))^{-1}, \quad x\in T\setminus\{o\},
\]
where $v<x$ if $v$ is on the unique path connecting $x$ to $o$ in $T$. Note that $g_p(x)^{-1}$ corresponds precisely to the probability that $x$ can be reached from $o$ in a realisation of bond percolation with edge retention probabilities given by $p=(p(e),e\in T)$. We call $g_p$ the gauge \emph{adapted} to the independent percolation $(p(e),e\in T)$.

Lyons' capacity criterion \cite{lyons_capacity_1992} now implies that
\begin{equation}\label{eq:precviacap}
\mathscr{C}_{g_p}(\partial T) \leq \P(\sharp C_o=\infty \text{ under the percolation } p)\leq 2 \mathscr{C}_{g_p}(\partial T),
\end{equation}
where \(C_o\) is the cluster of \(o\) obtained by performing
\(p\)-Bernoulli percolation on \(T\). Thus our strategy is to use a coupling
idea analogous to elementary Galton--Watson branching arguments, but to replace
mean offspring numbers by capacities of boundary sets. For a fixed tree, these
capacities are the quantities that determine whether the relevant independent
percolation survives.

In the particular case where $p(e)\equiv p_\beta= 1-e^{-\beta}$, i.e.\ the percolation process in question is $\mathsf{link}_\beta[T]$, we need to consider the family of \emph{purely exponential gauge functions} $\{g(x)=r^{|x|}, r>1\}$, which we parametrise as $\bar{q}(x)=q^{-|x|}$ with $q\in(0,1)$ for notational convenience. This yields the characterisation
\[
\beta_c^{\mathsf{link}}(T)=\inf\{\beta\geq 0: \mathscr{C}_{\bar{p}_\beta}(\partial T)>0 \}
\]
of the link threshold in terms of boundary capacity. Another, sometimes more
convenient, characterisation of this critical value is via the
\emph{branching number} of \(T\):
\[
\operatorname{br}(T):=\sup\{1/q: \mathscr{C}_{\bar{q}}(\partial T)>0,  q\in(0,1)\}=(1-\textup{e}^{-\beta_c^{\mathsf{link}}(T)})^{-1}.
\]
The branching number is also related to recurrence and transience of biased
random walks on \(T\) through the electrical network interpretation. Lyons
\cite{lyons_percolation_1990} further established the relation
\begin{equation}\label{eq:dim-br}
\dim(\partial T)=\log (\operatorname{br}(T))
\end{equation}
for all infinite trees, which is the general version of Hawkes' theorem.

Positive \(g_p\)-capacity is the potential-theoretic form of supercriticality
for independent percolation with retention probabilities \(p(e)\). By the
definition of capacity,
\begin{equation}\label{eq:dirichlet}
\mathscr{C}_{g_p}(\partial T)>0
\quad\text{if and only if}\quad
\exists \nu\in\Pi(\partial T):\mathscr{E}_{g_p}(\nu)<\infty.
\end{equation}
When the capacity is positive, Dirichlet's principle gives a unique energy
minimiser, called the \emph{harmonic measure} of $g_p$ (or of
$(p(e),e\in T)$). It corresponds to the hitting distribution at $\partial T$ for
the weighted random walk on $T$ that starts at $o$ and is defined by the
transition probabilities
\[
\pi(a(x),x)=\frac{1/\Delta g_p(e(x))}{1/\Delta g_p(e(a(x))) + \sum_{y: a(y)=a(x)} 1/\Delta g_p(e(y)) },
\]
for \(x\neq o\), with the parent-edge term in the denominator omitted when
\(a(x)=o\). This boundary hitting distribution exists exactly when the walk is
transient.

For a Borel set \(A\subseteq\partial T\), define its \emph{relative branching
number} relative to \(T\) by
\[
\operatorname{br}_T(A):=\sup\{1/q:q\in(0,1),\ \mathscr{C}_{\bar q}(A)>0\}.
\]
For a set \(G\subseteq T\setminus\{o\}\) and \(a>0\), let
\[
B_a(G):=
\left\{\xi\in\partial T:
\liminf_{n\to\infty}\frac1n\#\{1\leq i\leq n:\xi_i\in G\}<a
\right\}.
\]
For a non-negative vertex weight \(\psi:T\setminus\{o\}\to[0,\infty)\), let
\[
B_a(\psi):=
\left\{\xi\in\partial T:
\liminf_{n\to\infty}\frac1n\sum_{i=1}^n\psi(\xi_i)<a
\right\}.
\]

We call the following deterministic statement the \emph{weighted pruning
criterion} behind Theorem~A.
\begin{theorem}[Weighted pruning criterion]\label{thm:weighted-pruning-criterion}
Fix \(u\in[0,1]\). Let \((T,o)\) be an infinite locally finite rooted tree with
\(0<\dim(\partial T)<\infty\), and put
\(\beta_0=\beta_c^\mathsf{link}(T)\). For \(\eta>0\), define a vertex weight
\(\psi_\eta:T\setminus\{o\}\to[0,\infty)\) as follows. If
\(|x|\not\equiv1\pmod 3\), or if \(x\) has no child, set
\(\delta_\eta(x)=0\) and \(\psi_\eta(x)=0\). If \(|x|\equiv1\pmod3\) and \(x\)
has at least one child, put
\[
D_T^+(x)=\max\{\deg_T(y):y\text{ is a child of }x\}
\]
and
\[
\delta_\eta(x)
=
\frac12
\min_{\substack{\lambda\in[\beta_0,\beta_0+\eta]\\
2\leq d\leq \deg_T(x)\\
1\leq d^\ast\leq D_T^+(x)}}
R_u(d,d^\ast,\lambda),
\qquad
\psi_\eta(x)=-\log(1-\delta_\eta(x)).
\]
Here \(R_u(d,d^\ast,\lambda)\) is the local pruning probability defined by
\eqref{eq:blockingprobabilities} for \(u>0\) and by
\eqref{eq:blockingprobabilities2} for \(u=0\). Thus
\(0\leq\delta_\eta(x)<1\), so that \(\psi_\eta(x)\) is well defined. Suppose
that there are \(\eta>0\) and \(a>0\) such that
\[
\operatorname{br}_T\left(
\left\{\xi\in\partial T:
\liminf_{n\to\infty}\frac1n
\sum_{i=1}^n\psi_\eta(\xi_i)<a
\right\}
\right)
<
\operatorname{br}(T).
\]
Then
\[
\beta_c^\mathsf{loop}(T)>\beta_c^\mathsf{link}(T).
\]
\end{theorem}
The strict inequality in the hypothesis is the \emph{weighted branching-gap
hypothesis}. The proof is given in Section~\ref{sec:weighted-pruning-proof}.

\begin{remark}[Bounded-degree criterion]\label{rem:bounded-degree-criterion}
The following \emph{bounded-degree criterion} implies the weighted
branching-gap hypothesis in Theorem~\ref{thm:weighted-pruning-criterion}.
Suppose that there are
\(M<\infty\) and \(a>0\) such that, for
\[
\begin{aligned}
G_{M,1}:=\{x\in T\setminus\{o\}:{}& |x|\equiv 1 \!\!\!\pmod 3,\,
\deg_T(x)\leq M,\\
&\text{and every neighbour of }x\text{ has degree at most }M\},
\end{aligned}
\]
the set \(B_a(G_{M,1})\) has relative branching number strictly smaller than
\(\operatorname{br}(T)\). Fix \(\eta>0\), let
\(\beta_0=\beta_c^\mathsf{link}(T)\), and let \(\psi_\eta\) be the weight in
Theorem~\ref{thm:weighted-pruning-criterion}. Increasing \(M\) if necessary,
assume \(M\geq2\). By Proposition~\ref{prop:obvious}, the constant
\[
\delta_0=
\frac12
\min_{\substack{\lambda\in[\beta_0,\beta_0+\eta]\\
2\leq d\leq M\\
1\leq d^\ast\leq M}}
R_u(d,d^\ast,\lambda)
\]
is positive. Along a boundary ray, every vertex of \(G_{M,1}\) lying on the ray
has a child, and \(\psi_\eta(x)\geq -\log(1-\delta_0)\) for such vertices.
Hence, with
\(c_0=-\log(1-\delta_0)\),
\[
B_{a c_0}(\psi_\eta)\subseteq B_a(G_{M,1}).
\]
Thus the weighted branching-gap hypothesis of
Theorem~\ref{thm:weighted-pruning-criterion} holds. This is the bounded-degree
form used below for Galton--Watson trees.
\end{remark}

The hypothesis in
Theorem~\ref{thm:weighted-pruning-criterion} is a capacity condition for the
actual pruning weights and the following large-deviation
criterion gives a convenient way to verify it when the boundary measure has
the displayed regularity. For a vertex \(x\), write
\(b_x\subseteq\partial T\) for the \emph{boundary cylinder} consisting of rays
passing through \(x\).

\begin{theorem}\label{thm:boundary-regular}
Fix \(u\in[0,1]\). Let \((T,o)\) be an infinite locally finite rooted tree with
\(0<\dim(\partial T)<\infty\), put \(b=\operatorname{br}(T)\), and let
\(\psi_\eta\) be the weight from
Theorem~\ref{thm:weighted-pruning-criterion} for some \(\eta>0\). Suppose
that there are \(a>0\), \(\kappa>0\), \(c>0\), and a Borel probability measure
\(\nu\) on \(\partial T\) such that
\[
\nu(b_x)\geq c b^{-|x|}
\]
for every vertex \(x\) with infinite descendants, and such that, for all
sufficiently large \(n\),
\[
\nu\left(\left\{\xi\in\partial T:
\sum_{i=1}^n\psi_\eta(\xi_i)<a n
\right\}\right)
\leq c^{-1}e^{-\kappa n}.
\]
Then
\[
\beta_c^\mathsf{loop}(T)>\beta_c^\mathsf{link}(T).
\]
\end{theorem}

In particular, Theorem~\ref{thm:boundary-regular} applies almost surely in
random rooted-tree models for which a maximal-dimension boundary measure
satisfies the displayed cylinder lower bound and the pruning large-deviation
estimate. This is the form one obtains, for example, in finite-state or
quasi-Bernoulli boundary models, and in product-type finite-range decorations
of a regular backbone, once the corresponding boundary process has a positive
average pruning weight and the associated exponential lower-tail bound.

The rest of this section develops the inputs used in the proof of
Theorem~\ref{thm:weighted-pruning-criterion} and verifies the
Galton--Watson case of the simpler bounded-degree branching-gap hypothesis.
\begin{theorem}\label{thm:percolationontrees}
Let $(T,o)$ be an infinite locally finite rooted tree with
\(\operatorname{br}(T)\in(1,\infty)\), and set
\(p_c=\operatorname{br}(T)^{-1}\). Let
\(\delta:T\setminus\{o\}\to[0,1)\) be a deterministic deletion profile, and put
\(\psi(x)=-\log(1-\delta(x))\). Suppose that there is an \(a>0\) such that
\[
\operatorname{br}_T(B_a(\psi))<\operatorname{br}(T).
\]
Then there exists \(p_\ast\in(p_c,1)\) such that for every
\(p\in(p_c,p_\ast)\) the independent inhomogeneous percolation with edge
retention probabilities
\[
q(e_x)=p\bigl(1-\delta(x)\bigr),\quad x\in T\setminus\{o\},
\]
has finite root component almost surely.
\end{theorem}
\begin{proof}
Choose constants \(0<\alpha<a\) and \(0<\gamma<\alpha\). Since
\(B_\alpha(\psi)\subseteq B_a(\psi)\), we have
\[
\operatorname{br}_T(B_\alpha(\psi))\leq \operatorname{br}_T(B_a(\psi))
<\operatorname{br}(T).
\]
Put \(B=B_\alpha(\psi)\) and \(A=\partial T\setminus B\).
Choose \(p_\ast\in(p_c,1)\) so close to \(p_c\) that
\[
p_\ast<\operatorname{br}_T(B)^{-1}
\quad\text{and}\quad
p_\ast e^{-\gamma}<p_c,
\]
where the first upper bound is interpreted as \(+\infty\) if
\(\operatorname{br}_T(B)=0\). Fix \(p\in(p_c,p_\ast)\).

Assume for contradiction that the root component of the $q$-percolation is
infinite with positive probability. By \eqref{eq:precviacap} and the definition
of capacity, there is a Borel probability measure $\nu$ on $\partial T$ with
finite $g_q$-energy, where
\[
g_q(x)=\prod_{o<y\leq x}q(e_y)^{-1}
=p^{-|x|}\exp\left(\sum_{o<y\leq x}\psi(y)\right).
\]
At least one of \(\nu(B)\) and \(\nu(A)\) is positive. If \(\nu(B)>0\), let
\(\nu_B\) be the normalised restriction of \(\nu\) to \(B\). Since
\(q(e)\leq p\), we have \(g_q(x)\geq \bar p(x)\) for every \(x\), and hence
\[
\mathscr E_{\bar p}(\nu_B)\leq \mathscr E_{g_q}(\nu_B)<\infty.
\]
Thus \(B\) has positive \(\bar p\)-capacity, so
\(\operatorname{br}_T(B)\geq 1/p\). This contradicts
\(p<p_\ast<\operatorname{br}_T(B)^{-1}\).

It remains to consider the case \(\nu(A)>0\).

For \(m\geq1\), define
\[
A_m=\left\{\xi\in A:
\sum_{i=1}^n\psi(\xi_i)\geq\gamma n
\text{ for every }n\geq m
\right\}.
\]
Since \(\gamma<\alpha\), we have \(A=\bigcup_m A_m\), so there is an \(m\) with
\(\nu(A_m)>0\). Let \(\nu_m\) be the normalised restriction of \(\nu\) to
\(A_m\). Put \(p'=pe^{-\gamma}<p_c\). If \(\xi,\eta\in A_m\) and
\(x=\xi\wedge\eta\) satisfies \(|x|\geq m\), then
\[
g_q(x)=p^{-|x|}
\exp\left(\sum_{o<y\leq x}\psi(y)\right)
\geq p^{-|x|}e^{\gamma |x|}
=(p')^{-|x|}=\bar p'(x).
\]
For the finitely many possible common ancestors with \(|x|<m\), local
finiteness and \(\delta(x)<1\) at each vertex give a positive constant \(c_m\)
such that
\(g_q(x)\geq c_m\bar p'(x)\) for all common ancestors of pairs of rays in
\(A_m\). Consequently
\[
\mathscr E_{\bar p'}(\nu_m)\leq c_m^{-1}\mathscr E_{g_q}(\nu_m)<\infty.
\]
Thus \(A_m\) has positive \(\bar p'\)-capacity, and therefore so does
\(\partial T\). This contradicts \(p'<p_c=\operatorname{br}(T)^{-1}\). The root
component of the \(q\)-percolation is therefore finite almost surely.
\end{proof}

\begin{prop}\label{prop:gw-branching-gap}
Let \(T\) be a Galton--Watson tree with offspring law \(\zeta\), offspring
variable \(Z\), and finite mean \(m=\E Z\in(1,\infty)\), conditioned on
survival. Fix \(r\in\{0,1,2\}\). Then there are \(M<\infty\) and \(a>0\)
such that, almost surely,
\[
\operatorname{br}_T(B_a(G_{M,r}))<\operatorname{br}(T)=m,
\]
where
\[
\begin{aligned}
G_{M,r}:=\{x\in T\setminus\{o\}:{}& |x|\equiv r \!\!\!\pmod 3,\,
\deg_T(x)\leq M,\\
&\text{and every neighbour of }x\text{ has degree at most }M\}.
\end{aligned}
\]
The same conclusion holds for the \emph{augmented Galton--Watson tree}
obtained by joining the roots of two independent copies by one extra edge and
rooting the resulting tree at one of these two original roots, on the event that
the augmented tree is infinite.
\end{prop}
\begin{proof}
By \cite[Corollary 5.10]{lyons_probability_2017},
\(\operatorname{br}(T)=m\) almost surely on survival. We prove a strict
relative branching-number bound for \(B_a(G_{M,r})\).

Let \(\widehat Z\) be the \emph{size-biased offspring variable},
\[
\P(\widehat Z=k)=\frac{k\zeta_k}{m},\qquad k\geq1.
\]
For \(x\) at level \(n\), write \(v_i(x)\) for the ancestor of \(x\) at level
\(i\). We use the following first-moment identity. If \(F\) is a
non-negative function of the offspring variables along the ancestral line of a
level-\(n\) vertex and of the offspring variables of the children of these
ancestors, then
\[
\E\sum_{|x|=n}F(x)=m^n\E_\ast F(\xi_n),
\]
where under \(\P_\ast\) the ancestral offspring variables
\(\widehat Z_0,\ldots,\widehat Z_{n-1}\) have the size-biased law, the chosen
\emph{spine child} is uniform among the children at each ancestral level, the
terminal offspring variable has the original law, and all \emph{off-spine}
child subtrees have the original Galton--Watson law. Indeed, for bounded
cylinder \(F\), expand the
left-hand side over the ordered child choices along the ancestral path. Each
ancestral level contributes one factor \(Z_i\); after normalising these factors
by \(m\), the ancestral offspring variables have the size-biased law and the
selected child index is uniform. Monotone convergence gives the identity for
non-negative \(F\).

Set
\[
F_M=\P(Z\leq M-1),\qquad
\widehat F_M=\P(\widehat Z\leq M-1).
\]
For non-root vertices, \(\deg_T(v)\leq M\) is equivalent to the offspring
variable at \(v\) being at most \(M-1\).
For an internal level \(i\), the event \(v_i(\xi_n)\in G_{M,r}\) depends only
on the spine offspring variables at levels \(i-1,i,i+1\) and on the off-spine
children of the level-\(i\) spine vertex. Hence, restricted to one residue
class modulo \(3\), these events use disjoint variables and are independent.
For all such internal levels the success probability is
\[
p_M=\widehat F_M^2
\E\left[
\mathbf 1_{\{\widehat Z\leq M-1\}}F_M^{\widehat Z-1}
\right].
\]
The two factors \(\widehat F_M\) control the parent spine vertex and the spine
child, while the expectation controls the current spine vertex and all its
non-spine children. Since \(Z<\infty\) almost surely and \(\E Z<\infty\),
\(\widehat Z<\infty\) almost surely, and dominated convergence gives
\(p_M\to1\) as \(M\to\infty\). Choose \(M\) so large that \(p_M>3/4\), and set
\(a=1/7\).

Let
\[
S_n=\#\{2\leq i\leq n-2:
i\equiv r \!\!\!\pmod 3,\ v_i(\xi_n)\in G_{M,r}\}.
\]
Under \(\P_\ast\), \(S_n\) is binomial with
\[
N_n=\#\{2\leq i\leq n-2:i\equiv r \!\!\!\pmod 3\}=n/3+O(1)
\]
trials and success probability \(p_M\). The constant \(1/6\) is
\(\frac12\cdot\frac13\): since \(a=1/7<1/6\), we have
\(a n<N_n/2\) for all sufficiently large \(n\). Together with \(p_M>3/4\), a
Chernoff bound gives constants \(C<\infty\) and \(c>0\) such that
\[
\P_\ast(S_n<a n)\leq C\exp(-cn)
\]
for every \(n\).

Define the bad level-\(n\) prefixes by
\[
L_n=\{x:|x|=n,
\#\{2\leq i\leq n-2:
i\equiv r \!\!\!\pmod 3,\ v_i(x)\in G_{M,r}\}<a n\}.
\]
The first-moment identity yields
\[
\E |L_n|\leq C m^n\exp(-cn).
\]
Choose \(\rho\) with \(\max\{1,m\exp(-c)\}<\rho<m\). Markov's inequality and
the Borel--Cantelli lemma imply that, almost surely,
\[
|L_n|\leq \rho^n
\]
for all sufficiently large \(n\). This almost-sure statement remains valid
after conditioning on survival.

Let \(B=B_a(G_{M,r})\). Every ray in \(B\) has its level-\(n\) prefix in
\(L_n\) for infinitely many \(n\). Put
\[
U_n=\bigcup_{x\in L_n} b_x.
\]
If \(\mu\) is a Borel probability measure on \(\partial T\) with \(\mu(B)=1\),
then Tonelli's theorem gives \(\sum_n\mu(U_n)=\infty\).

Fix \(q<1/\rho\). For the exponential gauge \(\bar q(x)=q^{-|x|}\),
\[
\mathscr E_{\bar q}(\mu)
=1+\sum_{n\geq1}(q^{-n}-q^{-(n-1)})
\sum_{|x|=n}\mu(b_x)^2 .
\]
For all large \(n\), the cylinders \((b_x)_{x\in L_n}\) are disjoint and
\[
\sum_{x\in L_n}\mu(b_x)^2
\geq \frac{\mu(U_n)^2}{|L_n|}
\geq \rho^{-n}\mu(U_n)^2 .
\]
Therefore
\[
\mathscr E_{\bar q}(\mu)
\geq (1-q)\sum_{n\geq N}(q\rho)^{-n}\mu(U_n)^2
\]
for a sufficiently large \(N\). Since \(q\rho<1\), Cauchy--Schwarz implies
that the last sum diverges whenever \(\sum_n\mu(U_n)=\infty\). Thus every
probability measure carried by \(B\) has infinite \(\bar q\)-energy, so
\(\mathscr C_{\bar q}(B)=0\) for every \(q<1/\rho\). Consequently
\[
\operatorname{br}_T(B_a(G_{M,r}))\leq\rho<m=\operatorname{br}(T).
\]

For the augmented Galton--Watson tree, the boundary is the finite union of the
boundaries of the surviving independent copies, up to finite shifts in level.
The extra joining edge changes the set \(G_{M,r}\) only at finitely many
vertices on each ray, and the residue class in the second copy is merely
shifted. The preceding proof is uniform in \(r\), and the finite shifts change
the densities by \(o(1)\). For a finite union of boundary pieces, positive
\(\bar q\)-capacity of the union is equivalent to positive \(\bar q\)-capacity
of at least one piece: one direction is monotonicity, while in the other
direction a probability measure of finite energy on the union gives, after
restricting to a piece of positive mass and renormalising, a finite-energy
measure on that piece. Hence the relative branching number of the augmented
boundary, and of its bad-ray set, is the maximum of the corresponding relative
branching numbers in the surviving copies. The same strict branching gap
therefore holds on the event that the augmented tree is infinite.
\end{proof}

\section{Proof of the weighted pruning criterion}\label{sec:weighted-pruning-proof}
We now use the machinery set up in the previous section to prove
Theorem~\ref{thm:weighted-pruning-criterion}.

\begin{proof}[Proof of Theorem~\ref{thm:weighted-pruning-criterion}]
By the dimensional hypothesis and \eqref{eq:dim-br} we have
\[
1<\operatorname{br}(T)<\infty .
\]
Put \(p_c=\operatorname{br}(T)^{-1}\) and
\(\beta_0=\beta_c^\mathsf{link}(T)=-\log(1-p_c)\).

Let \(\eta>0\), \(a>0\), and the profiles \(\delta_\eta,\psi_\eta\) be as in
the branching-gap hypothesis of
Theorem~\ref{thm:weighted-pruning-criterion}. By construction,
\(\delta_\eta:T\setminus\{o\}\to[0,1)\).

Let \(\beta\in(\beta_0,\beta_0+\eta]\) and set \(p_\beta=1-\exp(-\beta)\). First
perform link percolation and denote the root link cluster by \(C_\beta(o)\).
Conditionally on \(C_\beta(o)\), consider a vertex
\(x\in C_\beta(o)\) with \(|x|\equiv1\pmod3\) and at least one child in
\(C_\beta(o)\). For such \(x\), set
\[
d_C(x)=\deg_{C_\beta(o)}(x),\qquad
d_C^\ast(x)=\min\{\deg_{C_\beta(o)}(y):y\text{ is a child of }x
\text{ in }C_\beta(o)\}.
\]
Then
\[
2\leq d_C(x)\leq \deg_T(x),
\]
and \(1\leq d_C^\ast(x)\leq D_T^+(x)\). Since
\(\beta\in[\beta_0,\beta_0+\eta]\), the parameter triple
\((d_C(x),d_C^\ast(x),\beta)\) lies in the minimisation range defining
\(\delta_\eta(x)\). Hence
\[
R_u(d_C(x),d_C^\ast(x),\beta)
\geq
\min_{\substack{\lambda\in[\beta_0,\beta_0+\eta]\\
2\leq d\leq \deg_T(x)\\
1\leq d^\ast\leq D_T^+(x)}}
R_u(d,d^\ast,\lambda)
=2\delta_\eta(x).
\]
This left-hand side is exactly the delayed pruning probability at \(x\),
computed inside \(C_\beta(o)\).

For each such \(x\), choose a child of minimal \(C_\beta(o)\)-degree, using an
extra independent mark to break ties. The pruning event attached to \(x\) uses
only the parent edge of \(x\), the chosen child edge, and the other edges
incident to \(x\) and to the chosen child; it does not use the other edges
incident to the parent \(a(x)\). If \(x,z\) are two distinct selected vertices,
then \(|x|\equiv |z|\equiv 1\pmod 3\), and the level-spacing check in the proof
of Lemma~\ref{lem:coupling1} shows that these finite edge sets are disjoint.
Moreover conditioning on \(C_\beta(o)\) is a product conditioning over the
edges, namely on the events that an edge has zero links or at least one link.
The corresponding pruning events are therefore conditionally independent. We
may consequently couple them with independent Bernoulli variables of parameters
\(\delta_\eta(x)\) so that a Bernoulli deletion at a selected vertex lying on
an infinite ray of \(C_\beta(o)\) forces the corresponding delayed-pruning
event.

It follows that if the independent two-stage process obtained by retaining
edges with probability \(p_\beta\) and then deleting each vertex \(x\) with
probability \(\delta_\eta(x)\) has finite root component, then so does
\(\mathsf{delay}_\beta\circ\mathsf{link}_\beta[T]\). Indeed, an infinite ray
surviving delayed pruning would also survive the auxiliary independent
two-stage thinning, up to finite dead ends which do not affect connection to
infinity. Theorem~\ref{thm:percolationontrees}, applied to
\(\delta_\eta\), gives a \(p_\ast>p_c\) such that the auxiliary two-stage
process has finite root component almost surely for every
\(p_\beta\in(p_c,p_\ast)\).
Replacing \(p_\ast\), if necessary, by a smaller value in
\((p_c,1-\exp(-\beta_0-\eta))\), and setting
\(\beta_\ast=-\log(1-p_\ast)\), we obtain
\[
\P\bigl(\mathsf{delay}_\beta\circ\mathsf{link}_\beta[T]
\text{ has an infinite root component}\bigr)=0
\quad\text{for all }\beta\in(\beta_0,\beta_\ast).
\]
Under the coupling in Lemma~\ref{lem:coupling1}, every infinite branch visited
by the root loop would give an infinite branch in
\(\mathsf{delay}_\beta\circ\mathsf{link}_\beta[T]\). Therefore the root loop is
finite almost surely throughout this interval. Hence
\(\beta_c^\mathsf{loop}(T)\geq \beta_\ast>\beta_0
=\beta_c^\mathsf{link}(T)\).
\end{proof}

\section{Proofs of the application theorems}\label{sec:applications-proofs}

\begin{proof}[Proof of Theorem~\ref{thm:boundary-regular}]
Let
\[
B=
\left\{\xi\in\partial T:
\liminf_{n\to\infty}\frac1n
\sum_{i=1}^n\psi_\eta(\xi_i)<a
\right\}.
\]
For \(n\geq1\), define
\[
L_n=\left\{x\in T: |x|=n,\,
\sum_{i=1}^n\psi_\eta(v_i(x))<a n
\right\},
\]
where \(v_i(x)\) is the ancestor of \(x\) at level \(i\). We may discard
vertices with no infinite descendants, since their boundary cylinders are
empty. The cylinders \((b_x)_{x\in L_n}\) are disjoint, and the hypotheses give
\[
c b^{-n}|L_n|
\leq
\sum_{x\in L_n}\nu(b_x)
\leq
c^{-1}e^{-\kappa n}
\]
for all sufficiently large \(n\). Hence
\[
|L_n|\leq c^{-2}(b e^{-\kappa})^n
\]
for all sufficiently large \(n\).

Every ray in \(B\) has prefixes in \(L_n\) for infinitely many \(n\). Put
\[
U_n=\bigcup_{x\in L_n}b_x .
\]
If \(\mu\) is a Borel probability measure on \(\partial T\) with
\(\mu(B)=1\), then Tonelli's theorem gives \(\sum_n\mu(U_n)=\infty\).

Fix \(\rho>b e^{-\kappa}\). For all large \(n\), we have
\(|L_n|\leq C\rho^n\) for some finite \(C\). Let \(q\in(0,1)\) satisfy
\(q\rho<1\). For the exponential gauge \(\bar q(x)=q^{-|x|}\),
\[
\mathscr E_{\bar q}(\mu)
\geq
(1-q)\sum_{n\geq N}q^{-n}
\sum_{x\in L_n}\mu(b_x)^2
\geq
\frac{1-q}{C}\sum_{n\geq N}(q\rho)^{-n}\mu(U_n)^2 .
\]
Since \(\sum_n(q\rho)^n<\infty\), Cauchy--Schwarz implies that the last sum
diverges whenever \(\sum_n\mu(U_n)=\infty\). Thus every probability measure
carried by \(B\) has infinite \(\bar q\)-energy for every \(q\in(0,1)\) with
\(q\rho<1\). Consequently \(\operatorname{br}_T(B)\leq\rho\). Letting
\(\rho\downarrow b e^{-\kappa}\) gives
\[
\operatorname{br}_T(B)\leq b e^{-\kappa}<b=\operatorname{br}(T).
\]
Therefore the weighted branching-gap hypothesis in
Theorem~\ref{thm:weighted-pruning-criterion} is satisfied, and
Theorem~\ref{thm:weighted-pruning-criterion} gives the claimed strict
inequality of critical values.
\end{proof}

\begin{proof}[Proof of Theorem~A]
Condition the Galton--Watson tree on survival. By
\cite[Corollary 5.10]{lyons_probability_2017},
\(\operatorname{br}(T)=m\) almost surely. Hence
\[
\beta_c^\mathsf{link}(T)=-\log(1-m^{-1})
\]
almost surely by the branching-number characterisation in
Section~\ref{sec:potential}. Proposition~\ref{prop:gw-branching-gap}, with
\(r=1\), verifies the bounded-degree criterion in
Remark~\ref{rem:bounded-degree-criterion}, and hence the weighted
branching-gap hypothesis in
Theorem~\ref{thm:weighted-pruning-criterion} for each fixed \(u\in[0,1]\). The
Galton--Watson event supplied by Proposition~\ref{prop:gw-branching-gap} does
not depend on \(u\). Since \(\dim(\partial T)=\log m\) by \eqref{eq:dim-br},
Theorem~\ref{thm:weighted-pruning-criterion} applies and gives
\[
\beta_c^\mathsf{loop}(T)>\beta_c^\mathsf{link}(T)
\]
almost surely on survival.
\end{proof}

\begin{proof}[Proof of Theorem~B]
Work in the crosses-only model and fix
\(\beta>0\). Under the annealed Galton--Watson law, the root link cluster is a
Galton--Watson tree with offspring variable
\[
Z_\beta\mid Z\sim\operatorname{Bin}(Z,1-e^{-\beta}).
\]
Define an auxiliary rooted subtree \(T'\) recursively. The root is present. If
a present vertex \(x\) has at least one child edge \(e_y\), \(a(y)=x\), that
carries more than one cross, then \(x\) is given no children in \(T'\).
Otherwise all children joined to \(x\) by at least one cross are retained in
\(T'\). Thus every edge of \(T'\) carries exactly one cross.

If \(T'\) is infinite, then the root loop is infinite. Indeed, in the
crosses-only representation a finite loop visits the whole vertical circle
over each vertex in its permutation cycle. Therefore, if the root loop were
finite and had vertex set \(W\), then \(W\) would be closed under all retained
one-cross child edges in \(T'\): when the loop visits \(x\in W\), it must meet
and cross the unique cross on every retained child edge of \(x\). No finite
rooted vertex set in an infinite locally finite rooted tree is closed under
all child edges, so \(W\) cannot be finite.

Conditionally on \(Z_\beta=n\), each of the \(n\) retained link edges carries
more than one cross with conditional probability
\[
h(\beta)
=
\frac{e^{-\beta}\sum_{k\geq2}\beta^k/k!}{1-e^{-\beta}}
=
\P(P>1\mid P>0),
\qquad P\sim\operatorname{Poisson}(\beta).
\]
Thus \(T'\) is a Galton--Watson tree with offspring variable
\[
Y_\beta=Z_\beta\mathbf 1_{\{B(Z_\beta,\beta)=0\}},
\]
where \(B(n,\beta)\sim\operatorname{Bin}(n,h(\beta))\). Let
\(f_\beta(z)=\E z^{Z_\beta}\). Then
\[
\E Y_\beta
=\E\left[Z_\beta(1-h(\beta))^{Z_\beta}\right]
=(1-h(\beta))f_\beta'(1-h(\beta)).
\]
Since
\[
f_\beta'(z)
=(1-e^{-\beta})
f'\bigl(e^{-\beta}+(1-e^{-\beta})z\bigr)
\]
and
\[
1-h(\beta)=\frac{\beta e^{-\beta}}{1-e^{-\beta}},
\]
we obtain
\[
\E Y_\beta
=
\beta e^{-\beta}f'\bigl((1+\beta)e^{-\beta}\bigr).
\]
Set \(\varepsilon_\beta=1-(1+\beta)e^{-\beta}\). Then
\[
\varepsilon_\beta=\frac{\beta^2}{2}+O(\beta^3),
\qquad
\frac{\beta e^{-\beta}}{\sqrt{\varepsilon_\beta}}\to\sqrt2 .
\]
The strict margin in the tail condition of Theorem~B therefore implies
\(\E Y_\beta>1\) for all sufficiently small \(\beta>0\). Hence \(T'\)
survives with positive annealed probability for all sufficiently small
\(\beta>0\), and so
\(\beta_{c,\mathrm{ann}}^\mathsf{loop}(\zeta)=0\). The same condition forces
\(m=\infty\), and then
\(\E Z_\beta=(1-e^{-\beta})m=\infty\) for every \(\beta>0\). The offspring law
of the root link cluster is not concentrated at \(0\), and this Galton--Watson
process therefore survives with positive probability for every \(\beta>0\). Thus
\(\beta_{c,\mathrm{ann}}^\mathsf{link}(\zeta)=0\).

If, up to normalisation, \(\zeta_k=k^{-\tau}L(k)\) for \(k\geq1\), with
\(\tau\in(1,3/2)\) and \(L\) slowly varying, then classical Tauberian
estimates, for instance \cite[XIII.5, Theorem 1]{feller2}, imply that
\(f'(1-\varepsilon)\) is regularly varying with index \(\tau-2<-1/2\).
Thus \(\sqrt{\varepsilon}f'(1-\varepsilon)\to\infty\), and the tail condition
of Theorem~B holds.
\end{proof}

\begin{proof}[Proof of the subcritical equality stated after Theorem~B]
Suppose \(m\leq1\). If \(m<1\), or if \(m=1\) and
\(\zeta\neq\delta_1\), the Galton--Watson tree dies out almost surely, and
both critical values are infinite for almost every realised tree. The
remaining case is \(\zeta=\delta_1\), where \(T\) is the one-way infinite line.
For every finite \(\beta\), link percolation on this line has retention
probability \(1-e^{-\beta}<1\), so the root link cluster is finite almost
surely. The root loop is contained in the root link cluster, and hence is
finite almost surely. Therefore both critical values are infinite.
\end{proof}

\noindent\textbf{\large Acknowledgements}. We thank Volker Betz for sharing his insights on the loop model and his helpful comments on the first draft of the manuscript. We further thank Daniel Ueltschi for inspiring discussions and hosting AK's visit to Warwick in February 2025.

\section*{References}
\renewcommand*{\bibfont}{\footnotesize}
\printbibliography[heading = none]

\footnotesize{

\noindent\textbf{Funding acknowledgement.} CM's research was partly funded by Deutsche Forschungsgemeinschaft (DFG, German Research Foundation) – SPP 2265 443916008. AK's research is funded by the Cusanuswerk e.V.
}

\end{document}